\documentclass[final,leqno,onefignum,onetabnum]{siamltex}
\pdfoutput=1

\usepackage{fullpage}						
\usepackage{textcomp}						
\usepackage{graphicx}						
\usepackage{subfig}							
\usepackage{caption}						
\usepackage{color}							

\usepackage{epstopdf}						
\usepackage{algorithm}						
\usepackage{enumitem}						
\usepackage{cancel}						    
\usepackage{amsmath, amsfonts, amssymb}		
\usepackage{amscd}						    
\usepackage[T1]{fontenc}                    
\usepackage{mathtools}						
\usepackage{tikz}							
\usepackage{cite}                           
\usepackage{siunitx}
\usepackage{caption,here}
\usepackage{textcomp}
\usetikzlibrary{decorations.pathmorphing,decorations.pathreplacing}

\usepackage{epstopdf}
\usetikzlibrary{decorations.pathmorphing,decorations.pathreplacing}

\newtheorem{remark}{Remark}

\usepackage[colorlinks=true,citecolor=blue,urlcolor=blue]{hyperref}		

\newcommand{\dt}{\Delta t}

\newcommand{\BigOh}{\mathcal O}

\title{Method of lines transpose: High order L-stable ${\mathcal O}(N)$ schemes for parabolic equations using successive convolution}

\author{
    Matthew F. Causley\thanks{
    Kettering University,
    Department of Mathematics, 
    1700 University Ave,
    Flint, Michigan, 48504, USA
    ({\tt mcausley@kettering.edu})
}
\and Hana Cho\thanks{
    Michigan State University,
    Department of Mathematics, 
    619 Red Cedar Rd.,
    East Lansing, Michigan, 48824, USA ({\tt chohana@msu.edu})
}
\and Andrew J. Christlieb\thanks{
    Michigan State University,
    Department of Mathematics and Electrical Engineering
    ({\tt andrewch@math.msu.edu})
}
\and David C. Seal\thanks{
    U.S. Naval Academy,
    Department of Mathematics, 
    121 Blake Road,
    Annapolis, MD 21402, USA
    ({\tt seal@usna.edu})
    }
}
\begin{document}

\maketitle
\slugger{sinum}{xxxx}{xx}{x}{x--x}

\begin{abstract}
We present a new solver for nonlinear parabolic problems that is L-stable
and achieves high order accuracy in space and time.  The solver is built by
first constructing a single-dimensional heat equation solver that 
uses fast $\BigOh(N)$ convolution.  This fundamental solver has arbitrary order of accuracy 
in space, and is based on the use of the Green's function 
to invert a modified Helmholtz equation.
Higher orders of accuracy in
time are then constructed through a novel technique known as successive
convolution (or resolvent expansions).
These resolvent expansions facilitate our proofs of stability and convergence,
and permit us to construct schemes that have provable stiff decay.
The multi-dimensional solver is built by repeated application of 
dimensionally split independent fundamental solvers.
Finally, we solve nonlinear parabolic problems by 
using the integrating factor method, where we apply the basic scheme to
invert linear terms (that look like a heat equation),
and make use of Hermite-Birkhoff interpolants to integrate the remaining 
nonlinear terms.  Our solver is applied to several linear and nonlinear
equations including heat, Allen-Cahn, and the Fitzhugh-Nagumo
system of equations in one and two dimensions.
\end{abstract}

%
%


		\section{Introduction}	\label{sec:introduction}

The prototypical parabolic differential equation is the heat equation.  It 
forms a cornerstone of mathematics and physics, and its understanding
is essential for defining more complicated mathematical models. 
Fourier introduced this equation as a means to describe transient heat flow.
Fick quickly recognized its importance to particle
and chemical concentrations. As a result, parabolic
equations are now ubiquitous in describing diffusion processes, which are
found in a vast array of physical problems, among which are reaction-diffusion
models of chemical kinetics \cite{Fisher1999,Fitzhugh1961,Nagumo1962}, phase
field models describing morphology and pattern formation in multiphase fluids
and solids \cite{Cahn1961,Cahn1971,Dai2013}, and even the volatility of stocks
and bonds in mathematical finance \cite{Shreve2004}.

Numerical solutions of (linear and nonlinear) diffusion equations have been
the subject of active research for many decades. As early as the 1950's
and 60's, it was recognized
that due to the parabolic scaling, method of lines discretizations of the heat
equation lead to numerically stiff systems of equations, especially for
explicit time stepping.
Larger time steps (on the order of the mesh spacing) can be taken with fully implicit solvers,
but in practice, full matrix inversions may become difficult and costly,
especially when memory is extremely limited, as was especially the case
for early computers.
%
Thus, alternate
dimensionally implicit (ADI) splitting methods
\cite{Douglas1955,Douglas1956,Peaceman1955,Douglas1962,Fairweather1967,Crank1996},
that make use of dimensional splitting and tridiagonal solvers, quickly gained
popularity as part of an effort to reduce the amount of memory required to
invert these systems.

Later on, memory constraints no longer defined the bottleneck for computing, and attention
shifted toward methods that focused on reducing floating point operations
(FLOPs), albeit with additional memory constraints. Most notable among these are Krylov methods
\cite{Lambers2008,Jia2008}, boundary integral methods \cite{Greengard1991,
Kropinski2011}, and quadrature methods
\cite{Lubich1992,Jiang2013,Kassam2005,Tausch2007,Li2009}. However, with the
advent of GPU processors, it appears that we are yet again seeing a paradigm
shift towards methods that should emphasize small memory footprints, even at the expense
of incurring a higher operation count. Thus, ADI-like methods, which can
efficiently decompose larger problems and limit overhead communication,
warrant further investigation, and these features are the motivating factor for this
work.

In this paper we propose a novel numerical method for obtaining solutions to
the linear heat equation, and nonlinear reaction-diffusion type equations.
As an alternative to classical MOL formulations, we use the method of lines
transpose (MOL$^T$), which is sometimes referred to as Rothe's method
\cite{Rothe32}, or the transverse method of lines \cite{Salazar2000}. In this
case, the PDE is first discretized in time, and the resulting semi-discrete
(modified Helmholtz) problem can be solved using a variety of methods. From
potential theory \cite{Jia2008, Kropinski2011}, the solution can be
constructed by discretizing boundary integrals. However, with dimensional
splitting (that is related to the original ADI formulations), the MOL$^T$ can be used to
analytically solve simpler, one-dimensional boundary value problems, and the
subsequent solution can be constructed through dimensional sweeps, resulting
in an $\BigOh(N\log N)$ \cite{Lyon2010,Bruno2010} or $\BigOh(N)$
\cite{Causley_2013,Causley_2013b,Causley_2013c} solver. Furthermore, we extend
the method to higher orders of accuracy by using a novel idea referred to as successive convolution. 
This strategy has recently been developed in \cite{Causley_2013} for the wave
equation by the present authors.  In the present work, we not only extend the method of
lines transpose to parabolic problems, but we recognize the resulting expansion as
a so-called resolvent expansion \cite{abadias2013c_0,grimm2010approximation},
which we leverage to prove stability and convergence of the successive convolution series. 
In addition, we incorporate nonlinear terms with an integrating factor
method that relies on high order Hermite-Birkhoff interpolants as well as  
the (linear) resolvent expansions developed in this paper.


The rest of this paper is organized as follows. In Section \ref{sec:op_calc} we
derive the basic scheme for the one-dimensional heat equation, which is
L-stable and can achieve high orders of accuracy in space and time.  
In Section \ref{sec:high-order-resolvent} we describe how to
obtain an arbitrary order discretization in a single dimension with resolvent
expansions.  In Section \ref{sec:multiple}, we describe how this can be
extended to multiple dimensions, and in Section \ref{sec:Linear} we present
results for linear heat in one and two dimensions.  In Section
\ref{sec:Source}, we describe how our approach can handle nonlinear source terms, and in Section 
\ref{sec:Numerical} we present numerous numerical results including Allen-Cahn
and the Fitzhugh-Nagumo system of equations.
Finally, we draw conclusions and point to future work in Section
\ref{sec:conclusions}.

       \section{First order scheme in one spatial dimension}	\label{sec:op_calc}

We begin by forming a semi-discrete solution to the 1D heat equation using the method of lines transpose (MOL$^T$). Let $u = u(x,t)$ satisfy
\begin{align}
   \label{eqn:heat}
   u_t &= \gamma u_{xx}, \quad (x,t) \in (a,b) \times [0,T],
\end{align}
with constant diffusion coefficient $\gamma$, and appropriate initial and
boundary conditions. The MOL$^T$ amounts to
employing a finite difference scheme for the time derivative, and collocating
the Laplacian term at time levels $t = t^{n}$ and $t=t^{n+1}$. Following a
similar approach from \cite{Causley_2013},
we introduce a free parameter $\beta>0$, so that the collocation has the 
form\footnote{In \cite{Causley_2013}, there are a total of two time
derivatives (and two space), so the right hand side depends on $u^{n+1}, u^n$, and $u^{n-1}$.}
\[
    \frac{u^{n+1}-u^n}{\Delta t} = \gamma\partial_{xx}\left(u^n+\frac{u^{n+1}-u^n}{\beta^2}\right), \qquad \beta>0.
\]
Next, we introduce the differential operator corresponding to the modified Helmholtz equation, defined by
\begin{equation}
    \label{eqn:MHL}
    \mathcal{L} = I- \frac{\partial_{xx}}{\alpha^2},	\quad
    \alpha = \frac{\beta}{\sqrt{\gamma\Delta t}}.
\end{equation}
After some algebra, we find that the scheme can be written as
\begin{equation}
    \label{eqn:First}
    \mathcal{L}[u^{n+1}-(1-\beta^2)u^n] = \beta^2 u^n.
\end{equation}
We note that there are at least two reasonable strategies for choosing $\beta$: 
\begin{enumerate}
    \item {\bf Maximize the order of accuracy.}  For example, if we choose $\beta^2=2$, then the discretization is the trapezoidal
    rule, which is second order accurate and A-stable.
    \item {\bf Enforce stiff decay.}  For example, if we choose $\beta^2=1$, then the discretization is the backward
    Euler scheme, which is first order accurate, L-stable, yet does not
    maximize the order of accuracy.
\end{enumerate}
Here and below, we opt for the second strategy, as the stiff decay of numerical solutions of the heat equation is of paramount importance. In
Section \ref{subsec:Stability}, we develop this discussion in the context
of higher order schemes that relies on a careful selection of $\beta$ as well
as repeated applications of a single inverse operator.

Upon solving equation \eqref{eqn:First} for $u^{n+1}$, we find that the
equation for the update is
\begin{equation}
	\label{eqn:First_update}
	u^{n+1} = (1-\beta^2)u^n + \beta^2\mathcal{L}^{-1}[u^n],
\end{equation}
that requires inverting a modified Helmholtz operator. We accomplish this \textit{analytically} by using Green's functions, from which
\begin{equation}
	\label{eqn:L_inverse}
	\mathcal{L}^{-1}[u] = \left(I- \frac{\partial_{xx}}{\alpha^2}\right)^{-1}[u] := \frac{\alpha}{2}\int_a^b e^{-\alpha|x-y|}u(y)dy + B_a e^{-\alpha(x-a)} + B_b e^{-\alpha(b-x)}.
\end{equation}
The coefficients $B_a$ and $B_b$ are determined by applying
prescribed boundary conditions at $x = a, b$ which we describe in Section
\ref{subsec:homogenous-solution}.

\begin{remark}
Alternatively, had we followed the method of lines (MOL) and first discretized \eqref{eqn:heat} in space,
then the differential operator  $\mathcal{L}$ would be replaced by an algebraic operator $L$, and would be
inverted \textit{numerically}.
\end{remark}

\begin{remark}
Although the update \eqref{eqn:First_update} (with $\beta \neq 1$) is only first order accurate, we describe in Section \ref{sec:high-order-resolvent}
how to extend our procedure to arbitrary order in time.
\end{remark}

This MOL$^T$ approach has several advantages. First, the solution is now explicit, but remains
unconditionally stable. Secondly, in recent work \cite{Causley_2013,Causley_2013b,Causley_2013c} we show that
the convolution integral in Eqn. \eqref{eqn:L_inverse} can be discretized
using a fast $\BigOh(N)$ algorithm,
where $N$ is the number of spatial points. We introduce more details in Section \ref{subsec:High_Space}, 
wherein we update the current algorithm to achieve a user-defined accuracy
of $\BigOh(\Delta x^{M})$ with mesh spacing $\Delta x$.
Finally, since the solution is still continuous in space, we can decouple the
spatial and temporal errors, and by combining resolvent expansions with
dimensional splitting, we extend the method to multiple dimensions without
recoupling the errors.

%

\begin{remark}
Since dimensional splitting is used, all spatial quantities are computed
according to a one-dimensional convolution integral of the form
\eqref{eqn:L_inverse}, which is performed on a line-by-line basis, following
so-called "dimensional sweeps". Since the discrete convolution is computed in
$\BigOh(N)$ complexity, the full solver scales linearly in the number of
spatial points (assuming each sweep is performed in parallel). 
\end{remark}

A fully discrete scheme is obtained after a spatial discretization of
\eqref{eqn:L_inverse}. The domain $[a,b]$ is partitioned into $N$ subdomains
$[x_{j-1},x_j]$, with $a = x_0<x_1< \ldots x_N = b$. The convolution operator
is comprised of a particular solution, which is defined by the convolution
integral 
\begin{equation}
   \label{eqn:I_def}
   I[u](x): = \frac{\alpha}{2} \int_a^b e^{-\alpha|x-y|}u(y) dy
\end{equation}
and a homogeneous solution
\begin{equation} \label{eqn:homogeneous_def}
	B_a e^{-\alpha(x-a)} + B_b e^{-\alpha(b-x)},
\end{equation}
both of which can be constructed in $\BigOh(N)$ operations using fast convolution.
We now describe each of these in turn, starting with the first.

       \subsection{Spatial discretization of the particular solution} \label{subsec:High_Space}

The particular solution is first split into $I[u](x) = I^L(x) + I^R(x)$, where
\[
   I^L(x) = \frac{\alpha}{2} \int_a^x e^{-\alpha(x-y)} u(y)  dy, \quad I^R(x) = \frac{\alpha}{2} \int_x^b e^{-\alpha(y-x)} u(y)  dy.
\]
Each of these parts independently satisfy the first order "initial value problems"
\begin{align} \notag
   (I^L)'(y) + \alpha I^L(y) = \frac{\alpha}{2} u(y)&, \quad a < y < x, \quad I^L(a) = 0, \\ \notag
   (I^R)'(y) - \alpha I^R(y) = -\frac{\alpha}{2} u(y)&, \quad x < y < b, \quad I^R(b) = 0, 
\end{align} 
where the prime denotes spatial differentiation. By symmetry, the scheme
for $I^R$ follows from that of $I^L$, which we describe. From the
integrating factor method, the integral satisfies the following identity,
known as exponential recursion
\[
   I^L(x_j) = e^{-\nu_j}I^L(x_{j-1}) + J^L(x_j), \quad \text{where} \quad J^L(x_j) = \frac{\alpha}{2} \int_{x_{j-1}}^{x_j}  e^{-\alpha(x_j-y)} u(y) dy,
\]
and
\[
   \nu_j = \alpha h_j, \qquad h_j = x_j-x_{j-1}.
\]
This expression is still exact, and indicates that only the "local" integral $J^L$ needs to be approximated. We therefore consider a projection of $u(y)$ onto $P_M$, the space of polynomials of degree $M$.  A local approximation
\[
   u(x_j-z h_j) \approx p_j(z), \quad z \in [0,1],
\]
is accurate to $\BigOh(h_j^{M})$, and defines a quadrature of the form
\begin{equation}
   \label{eqn:JL}
   J^L(x_j)  = \frac{\nu_j}{2} \int_{0}^{1} e^{-\nu_j z} u(x_j - h_jz) dz\approx \frac{\nu_j}{2} \int_0^1 e^{-\nu_j z} p_j(z) dz.
\end{equation}
If standard Lagrange interpolation is used, then the polynomials can be factorized as
\begin{equation}
   \label{eqn:pj}
   p_j(z) = \sum_{k=-\ell}^r p_{jk}(z) u_{j+k}  = z^T A_j^{-1} u^M_j,
\end{equation}
where $z = [1,z,\ldots, z^M]^T$, and $u^M_j = [u_{j-\ell},\ldots, u_j, \ldots, u_{j+r}]^T$, and $A_j$ is the Vandermonde matrix corresponding to the points $x_{j+k}$, for $k=-\ell \ldots r$. The values of $\ell$ and $r$ are such that $\ell+r=M+1$, and are centered about $j$ except near the boundaries, where a one-sided stencil is required.

Substituting this factorization into \eqref{eqn:JL} and integrating against an exponential, we find that
\[
   J^L(x_j) \approx J^L_j := \sum_{k=-\ell}^r w_{jk} u_{j+k},
\]
where the weights $w_j = [w_{j,-\ell},\ldots w_{j,r} ]$ satisfy
\begin{equation}
   \label{eqn:weights}
   w_{j}^T = \phi_j^T A_j^{-1}
\end{equation}
and where
\[
   \phi_{jk} = \frac{\nu_j}{2}\int_0^1 e^{-\nu_j z} z^k dz = \frac{k!e^{-\nu_j}}{2\nu_j}\left(e^{\nu_j}-\sum_{p=0}^k \frac{(\nu_j)^p}{p!}\right).
\]
If the weights are pre-computed, then the fast convolution algorithm scales as
$\BigOh(MN)$ per time step, and achieves a user-defined $\BigOh(M)$ in space. In
every example shown in this work, we choose $M=2$ or $M=4$.

\subsection{Homogeneous solution}
\label{subsec:homogenous-solution}

The homogeneous solution in \eqref{eqn:homogeneous_def} is used to enforce boundary
conditions. We first observe that all dependence on $x$ in the convolution
integral, $I[x] := I[u^n](x)$, in  \eqref{eqn:I_def} is on the Green's function,
which is a simple exponential function. Through direct differentiation, we obtain
%
\begin{equation} \label{eqn:derivative_integral}
    I_x(a) = \alpha I(a), \quad I_x(b) = -\alpha I(b).
\end{equation}
Various boundary conditions at $x=a$ and $x=b$ can be enforced by solving a
simple $2\times2$ system for the unknowns $B_a$ and $B_b$. For example, for periodic boundary
conditions we assume that (at each discrete time step, $t=t^n$)
\begin{equation} \label{eqn:periodic_assumption}
    u^n (a) = u^n (b), \quad u_x^n(a) = u_x^n(b), \qquad  \forall n \in \mathbb{N}.
\end{equation}
We next enforce this assumption to hold on the scheme \eqref{eqn:L_inverse},
\begin{align} \notag
\mathcal{L}^{-1}[u^n](a)  = \mathcal{L}^{-1}[u^n](b) &\quad \Longleftrightarrow \quad I(a) + B_a + B_b\mu = I(b) + B_a\mu + B_b, \\ \notag 
\mathcal{L}^{-1}_x[u^n](a)  = \mathcal{L}^{-1}_x[u^n](b) &\quad \Longleftrightarrow \quad \alpha \left(I(a) - B_a + B_b\mu \right)= \alpha \left( -I(b) - B_a\mu + B_b \right),
\end{align}
where $\mu = e^{-\alpha(b-a)}$ and the identities \eqref{eqn:derivative_integral} are used to find $\mathcal{L}_x^{-1}$. Solving this linear system yields
\begin{equation} \label{eqn:periodic_coefficient}
B_a = \dfrac{I(b)}{1- \mu}, \quad B_b = \dfrac{I(a)}{1- \mu}.
\end{equation}
Different boundary conditions (e.g. Neumann) follow an analogous
procedure that requires solving a simple $2\times2$ linear system for $B_a$ and
$B_b$.

        \section{Higher order schemes from resolvent expansions}
        \label{sec:high-order-resolvent}
In our recent work \cite{Causley_2013}, we apply a successive convolution approach 
to derive high order A-stable solvers for the wave equation. 
The key idea is to recognize the fact that, in view of
the modified Helmholtz operator \eqref{eqn:MHL}, the second derivative can be
factored as
\begin{equation}
    \label{eqn:Lap_sc}
    \left(-\frac{\partial_{xx}}{\alpha^2}\right) = \mathcal{L}-I = \mathcal{L}\left(I-\mathcal{L}^{-1}\right) :=\mathcal{L}\mathcal{D},
\end{equation}
where
\begin{equation}
    \label{eqn:D_def}
    \mathcal{D} = I-\mathcal{L}^{-1}, \qquad \mathcal{L} = \left(I-\mathcal{D}\right)^{-1}.
\end{equation}
Substitution of the second expression into \eqref{eqn:Lap_sc} determines the second derivative completely in terms of
this new operator
\begin{equation}
    \label{eqn:Lap_D}
    \left(-\frac{\partial_{xx}}{\alpha^2}\right) =
    \left(I-\mathcal{D}\right)^{-1}\mathcal{D} = \sum_{p=0}^\infty \mathcal{D}^p.
\end{equation}
This shows that second order partial derivatives of a sufficiently smooth
function $u(x)$ can be approximated by truncating a resolvent expansion based
on successively applying $\mathcal{D}$ to $u(x)$, which is a linear
combination of successive convolutions \eqref{eqn:L_inverse}.


Now, we consider a solution $u(x,t)$ to the heat equation \eqref{eqn:heat},
that for simplicity we take to be infinitely smooth.  We perform a Taylor expansion
on $u(x,t+\Delta t)$, and then use the Cauchy-Kovalevskaya procedure
\cite{seal2014picard,seal2014high} to exchange temporal and spatial
derivatives to yield
\begin{equation}
    \label{eqn:Taylor_t}
    u(x,t+\Delta t) = \sum_{p=0}^\infty \frac{(\Delta t \partial_t)^p}{p!}u(x,t) = \sum_{p=0}^\infty \frac{(\gamma \Delta t \partial_{xx})^p}{p!}u(x,t) =: e^{\gamma\Delta t \partial_{xx}} u(x,t).
\end{equation}

The  term $e^{\gamma\Delta t \partial_{xx}}$ is a spatial pseudo-differential
operator, and it compactly expresses the full Taylor series. Our goal is to make
use of the formula \eqref{eqn:Lap_D} to convert the Taylor series into a
resolvent expansion.  This can be performed term-by-term, and requires
rearranging a doubly infinite sum. However, if we instead work directly with
the operator defining the Taylor series, then
\[
    e^{\gamma \Delta t \partial_{xx} } =  e^{-\beta^2\left(-\frac{\partial_{xx}}{\alpha^2}\right)} =e^{-\beta^2\left(I-\mathcal{D}\right)^{-1}\mathcal{D}}.
\]
At first glance this expression looks quite unwieldy. However, fortune is on our side, since the generating
function of the \textit{generalized Laguerre} polynomials
$L^{(\lambda)}_p(z)$ is
\begin{equation}
    \label{eqn:Laguerre0}
    \sum_{p=0}^\infty  L^{(\lambda)}_p(z) t^p = \frac{1}{(1-t)^{\lambda+1}} e^{- \frac{t z}{1-t} },
\end{equation}
which bears a striking resemblance to our expansion. Indeed, if we take $\lambda=-1$, substitute $z=\beta^2$ and $t = \mathcal{D}$, then
\begin{equation}
    \label{eqn:Laguerre}
    e^{-\beta^2\left(I-\mathcal{D}\right)^{-1}\mathcal{D}} = \sum_{p=0}^\infty  L^{(-1)}_p(\beta^2) D^p
    = I+\sum_{p=1}^\infty L^{(-1)}_p(\beta^2) D^p.
\end{equation}

\subsection{Convergence}

This expansion has been considered in the context of $C_0-$ semigroups
\cite{abadias2013c_0,grimm2010approximation}, where $\left(-\frac{\partial_{xx}}{\alpha^2}\right)$
is replaced with a general closed operator $A$ on a Hilbert space $X$. In our notation, we
restate part $(ii)$ of Theorem 4.3 in \cite{abadias2013c_0}, which is proven therein.
\begin{theorem}
    Let the $C_0-$ semigroup
    \begin{align*}
        T(\beta^2)&=e^{-\beta^2\left(-\frac{\partial_{xx}}{\alpha^2}\right)} 
        = \sum_{p=0}^\infty L^{(-1)}_p(\beta^2) \left(I-\frac{\partial_{xx}}{\alpha^2}\right)^{-p} \left(-\frac{\partial_{xx}}{\alpha^2}\right)^p \\
        &= \sum_{p=0}^\infty L^{(-1)}_p(\beta^2) \mathcal{L}^{-p} \left(\mathcal{L}-I\right)^p 
         =\sum_{p=0}^\infty L^{(-1)}_p(\beta^2) D^p
    \end{align*}
    be approximated by
    \[
        T_P(\beta^2)= \sum_{p=0}^P L^{(-1)}_p(\beta^2) D^p.
    \]
    Then, for $u(x) \in C^{2P+2}$, there exists for each $\beta^2>0$ an integer $m_0$ such that for all integers $2\leq k \leq P+1$, with $P \geq m_0$,
    \[
        \left\|T(\beta^2)u - T_P(\beta^2) u\right\|\leq \frac{C(\beta^2,k)}{P^{k/2-1}}\left\|\left(-\frac{\partial_{xx}}{\alpha^2}\right)^k u\right\|,
    \]
where $C(\beta^2,k)$ is a constant that depends only on $\beta^2$ and $k$.
\end{theorem}
\begin{remark}
The salient point of the theorem is that, in consideration of $\alpha$ 
(c.f.  Eqn. \eqref{eqn:MHL}), the approximation error is of the form
$C \Delta t^{P+1} \left\|u^{(2P+2)}(x)\right\|$, which matches the form given by a typical Taylor method.
\end{remark}

Finally, we truncate the resolvent expansion \eqref{eqn:Laguerre} at order
$p=P$.  For the heat equation, this defines the numerical method as
\begin{equation}
    \label{eqn:Scheme_D}
    u(x,t+\Delta t) = u(x,t)+\sum_{p=1}^P L^{(-1)}_p(\beta^2) \mathcal{D}^p[u](x,t),
\end{equation}
which has a truncation error of the form
\begin{equation}
    \label{eqn:Trunc}
    \tau:=  L^{(-1)}_{P+1}(\beta^2) \mathcal{D}^{P+1}[u](x,t) + \BigOh(\Delta t^{P+2}).
\end{equation}
For $P=1,2,3$, these schemes (evaluated at $t = t^n$) are
\begin{align}
	u^{n+1} &= u^n - \beta^2\mathcal{D}[u^n], \\
	u^{n+1} &= u^n - \beta^2\mathcal{D}[u^n]-\left(\beta^2-\frac{\beta^4}{2}\right)\mathcal{D}^2[u^n], \\
	u^{n+1} &= u^n - \beta^2\mathcal{D}[u^n]-\left(\beta^2-\frac{\beta^4}{2}\right)\mathcal{D}^2[u^n]-\left(\beta^2-\beta^4+\frac{\beta^6}{6}\right)\mathcal{D}^3[u^n].
\end{align}
We note that for implementation, each operator is applied successively, and is defined by
\[
    \mathcal{D}^{(p+1)}[u] := \mathcal{D}[\mathcal{D}^{p}[u]], \quad \mathcal{D}^{0}[u] := u.
\]

\subsection{Stability}
\label{subsec:Stability}

%
There remains one critical issue: the choice of the free parameter $\beta$. 
In 1974, N{\o}rsett studied a similar single-step multiderivative
method for the heat equation \cite{Norsett1974}
and he too, had a free parameter in his solver.
We follow his lead on the Von-Neumann analysis based on his MOL
discretization, but in this work we optimize $\beta$ to obtain stiff decay,
whereas N{\o}rsett chose $\beta$ to maximize the order of accuracy of the
solver.

Consider the linear test problem
\[
    \frac{dy}{dt} = \lambda y, \quad \quad y(0) = 1, \quad \lambda \in \mathbb{C}, 
\]
whose exact solution $y(t)$ satisfies
\[
    y(t+h) = e^z y(t), \quad z = h\lambda \in \mathbb{C}.
\]
Application of \eqref{eqn:Scheme_D} to this test problem results in
\[
    y(t+h) = \sum_{p=0}^P L^{(-1)}_p(\beta^2) \hat{D}^p y(t) = \phi(z)y(t)
\]
where
\[
    \hat{D} = \frac{-(z/\beta^2)}{1-(z/\beta^2)} = 1-\left(1-(z/\beta^2)\right)^{-1}.
\]
The generalized Laguerre polynomials satisfy many identities, the following of
which is the most pertinent:
\begin{equation}
     \label{eqn:Laguerre_Identity}
     L^{(0)}_{p+1}(x)-L^{(0)}_{p}(x) = L^{(-1)}_{p+1}(x) = \left(\frac{x}{p+1}\right)\frac{d}{dx}L^{(0)}_{p+1}(x).
\end{equation}
Here, $L^{(0)}_p(x)$ is the standard Laguerre polynomial $L_p(x)$.
Following standard definitions, we say that a numerical scheme is
\textit{$A$-stable}, provided $|\phi|\leq 1$ in the left-half of the complex
plane $z\in \mathbb{C}^-$. Likewise, a scheme exhibits \textit{stiff decay} if
$\phi(z) \to 0$ as $Re(z) \to -\infty$. If an $A$-stable method also exhibits
stiff decay, it is $L$-stable.

Now, observing that $\hat{D} \to 1$ as $Re(z) \to -\infty$, we find that
\begin{equation}
    \label{eqn:limit}
    \lim_{z\to -\infty} \phi(z)  = \sum_{p=0}^P L^{(-1)}_p(\beta^2) = L^{0}_0(\beta^2)+\sum_{p=1}^P \left(L^{0}_p(\beta^2)-L^{0}_{p-1}(\beta^2)\right) =L^{0}_P(\beta^2),
\end{equation}
where we have used the first two expressions in \eqref{eqn:Laguerre_Identity}
to introduce a telescoping sum. We are now prepared to prove the following.
\begin{theorem}
    Let $u(x,t)$ be an approximate solution to the heat equation \eqref{eqn:heat}, given by the successive convolution scheme \eqref{eqn:Scheme_D}. Then,
    \begin{enumerate}
        \item If $\beta^2=x_1^{(P)}$ is chosen as the smallest root of $L'_{P+1}(x) = (L^{(0)}_{P+1}(x))'$, then the scheme achieves order $P+1$, but does not exhibit stiff decay.
        \item If $\beta^2 = x_1^{(P)}$ is chosen as the smallest root of $L_{P}(x) = L^{(0)}_{P}(x)$, then the scheme achieves order $P$, and exhibits stiff decay.
        \item Following the first strategy, the schemes are $A$-stable for $P=1,2,3$, whereas the second strategy ensures $L$-stability. For both strategies, $A(\alpha)$-stability is achieved for $P>3$, with large values of $\alpha \approx \pi/2$.
    \end{enumerate}
\end{theorem}

\begin{proof}
The proof follows by applying the maximum modulus principle coupled with
\eqref{eqn:limit}. For part 1, upon examining the truncation error
\eqref{eqn:Trunc}, we see that an additional order of accuracy can be
gained if we choose
\[
    L^{(-1)}_{P+1}(\beta^2) = \left(\frac{\beta^2}{P+1}\right)L'_{P+1}(\beta^2) = 0.
\]
However, $L_{P}(\beta^2)\neq 0$ for this choice, and so stiff decay does not
hold. For part 2, we instead enforce stiff decay, but then the truncation
error is of order $P$. Finally, part 3 is demonstrated by the maximum
amplification factors $\phi$ along the imaginary axis, as shown for both
strategies in Figure \ref{fig:Stab}. In particular, we observe that
$|\phi(iy)|\leq 1$ for $P=1,2,3$.
\end{proof}

\begin{figure}[h!]
    \subfloat[Maximal Order]{\includegraphics[width=0.48\linewidth]{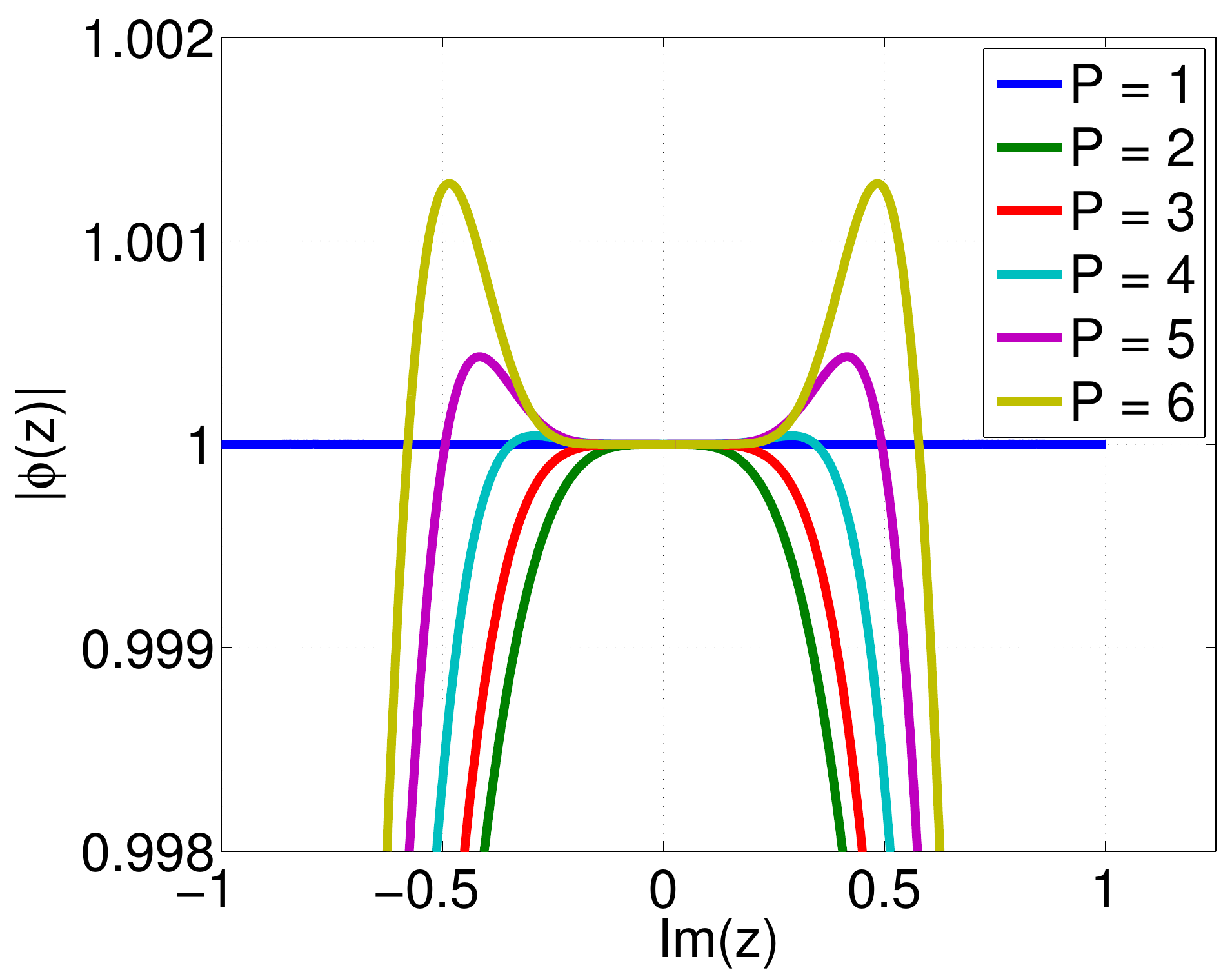}}
    \hspace{.02\linewidth}
    \subfloat[Stiff Decay]{\includegraphics[width=0.48\linewidth]{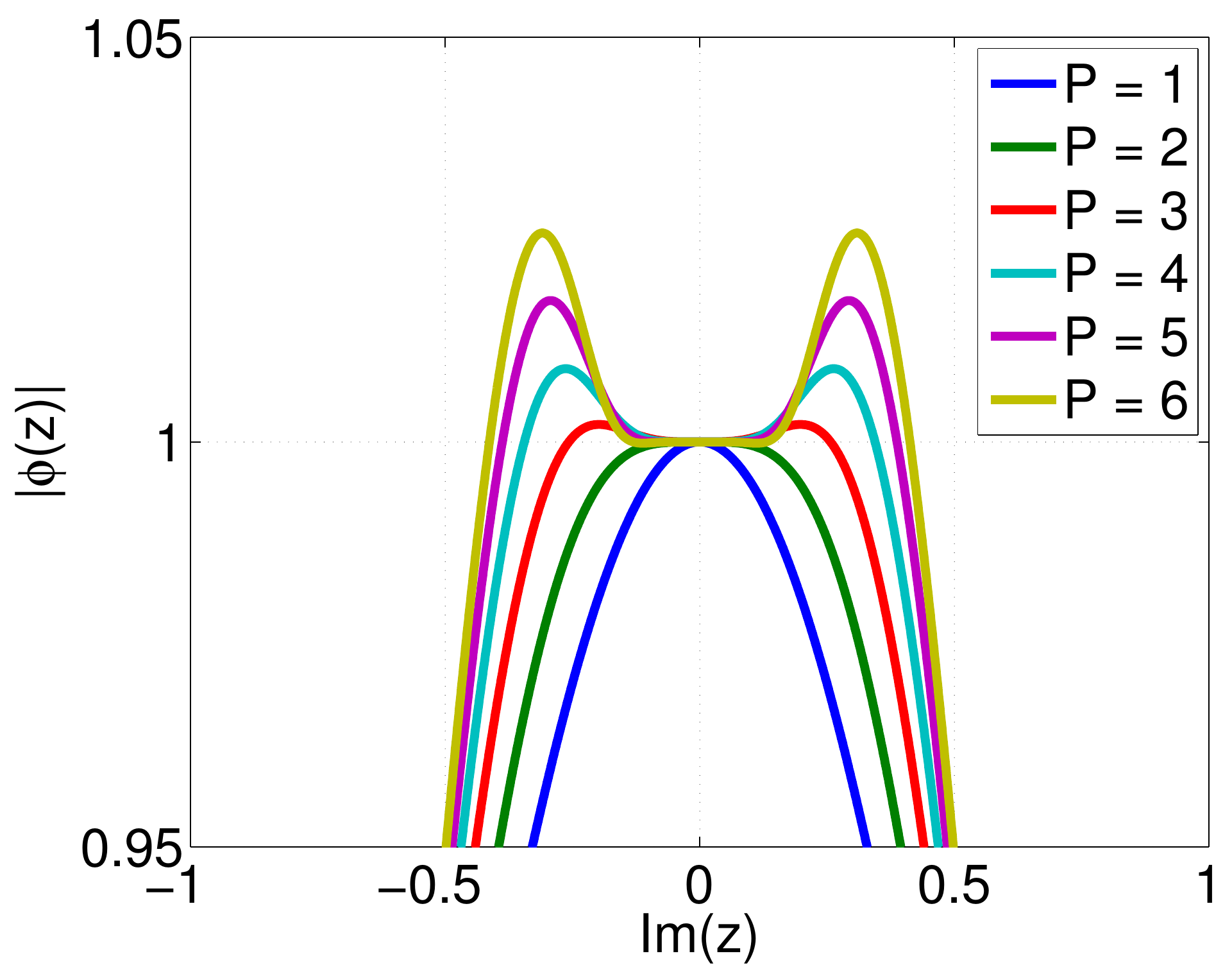}}
    \caption{Maximum amplification factors $|\phi(iy)|$ for the first few
    orders $P$, with (a) maximal order, or (b) stiff decay. When maximizing order, the first 3 schemes exhibit $A$-stability,
    whereas ensuring stiff decay leads to $L$-stable schemes. For $P>3$, both schemes become $A(\alpha)$-stable.
    }
    \label{fig:Stab}
\end{figure}
\begin{remark}
    In \cite{Norsett1974}, the scheme was chosen to maximize the order of accuracy, implicitly leading to eliminating the first term in the truncation error \eqref{eqn:Trunc}, which is equivalent to the first strategy. However, in this work we follow the second strategy, and choose $\beta^2$ as the smallest root of $L_P(x)$ to ensure stiff decay.
\end{remark}

For comparison we record the values of $\beta^2$ chosen for each order $1\leq
P \leq 6$, to those of N{\o}rsett in Table \ref{tab:LagZeros}.  For all of our
solvers, we choose $\beta$ to be the largest possible value that still yields
provable stiff decay.

\begin{table}[htbp]
\begin{center}
        \caption{Values of $\beta^2$ chosen for orders $P=1,2,\ldots 6$. The
        first column are those used in our schemes, and uniquely guarantee
        stiff decay and $A(0)$-stability. For comparison, we also display the
        values in N{\o}rsett \cite{Norsett1974} which give optimal order
        $P+1$, at the expense of stiff decay.  \label{tab:LagZeros}
        }
        \begin{tabular}{|c||c|c||c|c|}
            \hline 
                & \multicolumn{2}{c||}{Stiff Decay}& \multicolumn{2}{c|}{Maximal Order} \\
            \hline
            $P$	& $\beta^2$ & $L_P(\beta^2)$       & $\beta^2$ & $L_P(\beta^2)$   	    \\
            \hline
            $1$	& 1.0000	& 0			           & 2.0000    & -1.0000                \\
            \hline
            $2$	& 0.5858	& 0			           & 1.2679    & -0.7320                \\
            \hline
            $3$	& 0.4158	& 0			           & 0.9358    & -0.6304                \\
            \hline
            $4$	& 0.3225	& 0			           & 0.7433    & -0.5768                \\
            \hline
            $5$	& 0.2636	& 0			           & 0.6170    & -0.5436                \\
            \hline
            $6$	& 0.2228	& 0			           & 0.5277    & -0.5211                \\
            \hline
        \end{tabular}
\end{center}
\end{table}

%
%

\section{Resolvent expansions for multiple spatial dimensions}
        \label{sec:multiple}

We extend the 1D solver to multiple spatial dimensions through the use of
dimensional splitting. Our key observation is that we can use the
factorization property of the exponential to perform the series expansion. For
instance, in three dimensions, we have
\begin{equation}
    e^{\gamma\Delta t \nabla^2} = e^{-\beta^2\left(-\frac{\partial_{xx}}{\alpha^2}\right)}e^{-\beta^2\left(-\frac{\partial_{yy}}{\alpha^2}\right)}e^{-\beta^2\left(-\frac{\partial_{zz}}{\alpha^2}\right)}.
\end{equation}
Now, we first replace each term with the identity \eqref{eqn:Laguerre}
dimension by dimension, and then truncate the expansions which will be in
terms of the univariate operators $\mathcal{L}_\gamma^{-1}$ and
$\mathcal{D}_\gamma$ for $\gamma = \{x,y,z\}$ as defined by \eqref{eqn:MHL},
and \eqref{eqn:Scheme_D} acting on a function $u^n(x,y,z)$. This infinite sum
with three indices must then be truncated to order $P$, and after a change of
indices we find
\begin{equation}
\label{eqn:EPL_3}
    E_P = \sum_{p,q,r} \binom{P-1}{p,q,r} L^{(-1)}_p(\beta^2)L^{(-1)}_q(\beta^2)L^{(-1)}_r(\beta^2)\mathcal{D}_x^p \mathcal{D}_y^q\mathcal{D}_z^r,
\end{equation}
in 3D, with the corresponding 2D operator given by
\begin{equation}
\label{eqn:EPL_2}
    E_P = \sum_{p,q} \binom{P-1}{p,q} L^{(-1)}_p(\beta^2)L^{(-1)}_q(\beta^2)\mathcal{D}_x^p \mathcal{D}_y^q.
\end{equation}
Here we adopt the notation that sums are taken over all non-negative indices
that sum to $P$, and the multinomial coefficients are defined such that
$\binom{n}{p,q,r} = \frac{n!}{p!q!r!}$ and $\binom{n}{p,q} = \frac{n!}{p!q!}$.

\begin{remark}
The proof of stability for the multi-dimensional algorithm follows directly
from that of the one-dimensional case, with the same approach applied to each
spatial dimension (i.e. $\phi(z) = \phi_x(z)\phi_y(z)$ for the 2D case, and
similarly for the 3D case).
\end{remark}


\section{The heat equation} \label{sec:Linear}

\subsection{Heat equation in 1D} 
\label{subsec:1dheat}

We first illustrate the accuracy of our method for the 1D heat equation defined
in \eqref{eqn:heat}. We consider initial conditions $u(x,0) = \sin(x)$, for $x
\in [0,2\pi]$ with periodic boundary conditions.  We integrate up to a final
time of $T=4$, and set $\gamma = 0.18^2$.  
We use the fast convolution algorithm that is fourth order accurate in space
($M=4$), and set the spatial grid size to be $\Delta x = \frac{2\pi}{1024} \approx 0.0061$.
This ensures that the dominant error in the solution is temporal.
We compute errors by the $L^{\infty}$-norm, and compare against the
exact solution $u(x,T) = e^{-\gamma T}u_0(x)$ at $T=4$. The result of a
temporal refinement study for $P=1,2$ and $3$ is presented in Table
\ref{tab:refinement_1D_1}. 

\begin{table}[htbp]
\begin{center}
\caption{Refinement study for a 1D Heat equation defined in \ref{subsec:1dheat}.}
    \label{tab:refinement_1D_1}
\begin{centering}
\begin{tabular}{|c||c|c||c|c||c|c|}
\hline 
& \multicolumn{2}{c||}{$P=1$}	& \multicolumn{2}{c||}{$P=2$ }            & \multicolumn{2}{c|}{$P=3$}			\\ \hline
$\Delta t$  & $L^{\infty}$ error    & order & 	$L^{\infty}$ error	      & order & 	$L^{\infty}$ error	& order	\\ \hline
$0.1$		& $\num{0.00018405}$	& $-$		& $\num{0.0000016255}$    & $-$		& $\num{0.000000024225}$	& $-$	 		\\ \hline
$0.05$		& $\num{0.000092121}$	& $0.9985$	& $\num{0.00000040841}$	  & $1.9928$	& $\num{0.0000000030620}$	& $2.9839$		\\ \hline
$0.025$		& $\num{0.000046084}$	& $0.9993$	& $\num{0.00000010236}$	  & $1.9964$	& $\num{0.00000000038501}$	& $2.9915$		\\ \hline
$0.0125$	& $\num{0.000023048}$	& $0.9996$	& $\num{0.000000025622}$  & $1.9982$	&$\num{0.000000000048402}$	& $2.9918$		\\ \hline
$0.00625$	& $\num{0.000011525}$	& $0.9998$	& $\num{0.0000000064097}$ & $1.9990$	& $\num{0.0000000000062021}$	& $2.9642$		\\ \hline
\end{tabular}
\end{centering}
\end{center}
\end{table}

\subsection{Heat equation in 2D}
\label{subsec:2dheat}

As a second example, we present results for the 2D heat equation. We consider initial
conditions $u(x,y,0)=\sin(x)\sin(y)$, for $(x,y) \in [0,2\pi] \times
[0,2\pi]$ with periodic boundary conditions.  We use a uniform mesh of size $\Delta x=\Delta y=
2\pi/512 \approx 0.0123$. Likewise, the $L^{\infty}$-error is computed by
comparing against the exact solution $u(x,y,T) = e^{-2\gamma T} u_0(x,y)$ at
the final time $T=1$. In
Table \ref{tab:refinement_2D_1}, we present results for a temporal refinement
study for orders $P=1,2,$ and $3$.

\begin{table}[htbp]
\begin{center}
\caption{Refinement study for a 2D Heat equation defined in \ref{subsec:2dheat}.}
    \label{tab:refinement_2D_1}
\begin{centering}
\begin{tabular}{|c||c|c||c|c||c|c|}
\hline 
        & \multicolumn{2}{c||}{$P=1$}	& \multicolumn{2}{c||}{$P=2$ }				& \multicolumn{2}{c|}{$P=3$}			\\ \hline
$\Delta t$		& 	 $L^{\infty}$-error		& order& 	 $L^{\infty}$-error		& order & 	 $L^{\infty}$-error	& order	\\ \hline
$0.1$		& $\num{0.000098182}$	& $-$		& $\num{0.00000086717}$	& $-$		& $\num{0.000000012925}$	& $-$	 		\\ \hline
$0.05$		& $\num{0.000049143}$	& $0.9985$	& $\num{0.00000021788}$	& $1.9928$	& $\num{0.0000000016354}$	& $2.9825$		\\ \hline
$0.025$		& $\num{0.000024584}$	& $0.9992$	& $\num{0.000000054608}$	& $1.9963$	& $\num{0.00000000020791}$	& $2.9756$		\\ \hline
$0.0125$	& $\num{0.000012295}$	& $0.9996$	& $\num{0.000000013672}$	& $1.9979$	&$\num{0.000000000029204}$	& $2.8317$		\\ \hline
\end{tabular}
\end{centering}
\end{center}
\end{table}


        \section{Reaction-diffusion systems}
        \label{sec:Source}

We next extend our method to nonlinear reaction-diffusion systems of the form
\begin{align}
    \label{eqn:Source}
    {\bf u}_t = {\bf D} \nabla^2 {\bf u} + {\bf{F}}({\bf {u}}), \quad ({\bf x},t) \in \Omega \times (0,T],
\end{align}
where ${\bf u} = (u_1, u_2, \cdots, u_N)$, with $u_i = u_i({\bf x},t)$, $\bf D$
is a diffusion coefficient matrix, and the reaction term ${\bf F} := (f_1, f_2,
\cdots, f_N)$ is a function of $u_i$, $(i=1,2,\cdots,N)$. In the above, 
$\Omega \subset {\mathbb R}^N$ is a bounded domain, and we assume appropriate initial values and boundary conditions.
We shall view the diffusion as being the linear part of the differential operator, and invert this linear part analytically, using successive convolution. To derive the scheme, we use operator calculus to first write
\begin{align}\label{eqn:operator_calculus}
    \left(\partial_t - {\bf D} \nabla^2\right){\bf u} = {\bf{F}} \quad \implies \quad \left(e^{-{\bf D} t \nabla^2} {\bf u}\right)_t = e^{-{\bf D} t \nabla^2}{\bf{F}},
\end{align}
where $e^{-{\bf D} t \nabla^2}$ is a pseudo-differential operator. Upon integrating \eqref{eqn:operator_calculus} over the interval $[t, t+\Delta t]$, we arrive at the update equation
\begin{align}
    {\bf u} (t+\Delta t) - e^{ {\bf D} \Delta t \nabla^2}{\bf u} (t) =& \int_t^{t+\Delta t} e^{ {\bf D} (t+\Delta t-\tau) \nabla^2 }{\bf F} (\tau) d\tau \nonumber  \\
    \label{eqn:first}
                                =& \int_0^{\Delta t} e^{{\bf D}(\Delta t- \tau) \nabla^2}{\bf F} (t+ \tau)d\tau,
\end{align}
where we have made use of the abbreviated notation, ${\bf F}(t) := {\bf F}({\bf u}({\bf x},t))$.
On the left hand side, the diffusion terms have been collected by this pseudo-differential operator, and will be approximated using the successive convolution techniques developed above. The reaction terms on the right hand side \eqref{eqn:first} are fully nonlinear, and we must consider nonlinear stability when choosing a method of discretization.


We first consider approximating the integral on the right hand side \eqref{eqn:first} with the trapezoidal rule. This defines a single-step update equation, which will be second order accurate
\begin{align} \label{eqn:2nd_scheme}
    {\bf u}(t+\Delta t) - e^{{\bf D} \Delta t \nabla^2}{\bf u}(t) = 
    \frac{\Delta t}{2} \left[ e^{{\bf D} \Delta t \nabla^2} {\bf F}(t) +{\bf F}(t+\Delta t) \right].
\end{align}
We may also obtain a single-step third order scheme, using multiderivative integration \cite{seal2014high}. By replacing the integrand \eqref{eqn:first} with a third order Hermite-Birkhoff interpolant and performing exact integration of the resulting
function, we arrive at
\begin{align} \label{eqn:3rd_scheme}
    {\bf u}(t+\Delta t) - e^{{\bf D} \Delta t \nabla^2}{\bf u}(t) = e^{{\bf D} \Delta t \nabla^2} 
    \left[ \dfrac{2 \Delta t}{3} {\bf F}(t) + \dfrac{\Delta t}{6} \left( -{\bf D}\Delta t \nabla^2 {\bf F}(t) + \Delta t \dfrac{d{\bf F}}{d{\bf t}}(t) \right)\right] 
    +\dfrac{\Delta t}{3}{\bf F}(t+\Delta t),
\end{align}
where $\frac{d{\bf F}}{dt}(t) = \frac{d{\bf F}}{d{\bf u}}(t) \cdot ({\bf D}\nabla^2 {\bf u}(t)+{\bf F}(t))$.
The Hermite-Birkhoff interpolant that matches the integrand in \eqref{eqn:first} at times $\tau = 0$, and $\tau=\dt$, as well as its derivative at time $\tau = 0$ produces the quadrature rule in \eqref{eqn:3rd_scheme}.

\begin{remark}
The proposed schemes in \eqref{eqn:2nd_scheme} and \eqref{eqn:3rd_scheme} produce nonlinear equations
for ${\bf u}(x,t+\Delta t)$ that need to be solved at each time step.  Therefore, any implicit solver 
will necessarily be problem dependent.
\end{remark}

For the problems examined in this work, we make use of simple fixed-point iterative schemes.  
We stabilize our iterative solvers by linearizing ${\bf F} ({\bf u})$ about a background 
state ${\bf F}_{\bf u}({\bf u}^{\ast})$, which depends on the problem under consideration. 

%

\subsection{A discretization of the Laplacian operator} 

Upon perusing the third order update equation \eqref{eqn:3rd_scheme}, we will need to use successive convolution to replace the psuedo-differential operator $\exp\left({\bf D} \Delta t \nabla^2\right)$, as well as the Laplacian operator $\nabla^2$. This latter point has been detailed in \cite{Causley_2013}, and so we comment briefly on it here. Using the one-dimensional expansion \eqref{eqn:Lap_D}, we observe that the two-dimensional Laplacian is similarly given by
\[
	-\frac{\nabla^2}{\alpha^2} = -\frac{\partial_{xx}}{\alpha^2}-\frac{\partial_{yy}}{\alpha^2} = \sum_{p=1}^\infty \left(\mathcal{D}_x^p+\mathcal{D}_y^p\right),
\]
and can be truncated at the appropriate accuracy $p=P$. Here, the subscripts indicate that the convolution is only in one spatial direction, and the other variable is held fixed. Thus, $\mathcal{D}_x$ is applied along horizontal lines for fixed $y$-values, and likewise for $\mathcal{D}_y$.

\section{Nonlinear numerical results}
\label{sec:Numerical}

\subsection{Allen-Cahn}
\label{subsec:allencahn}

We examine in greater detail the application of our schem to the Allen-Cahn (AC) equation \cite{Allen1979},
\begin{align}
    \label{eqn:AC}
    u_t = \epsilon^2 \nabla^2 u + f(u), \qquad (x,t) \in \Omega \times (0,T],
\end{align}
where the reaction term is $f(u) = u - u^3$, and $\Omega \subset {\mathbb R}^d$ is a bounded
domain, and $u$ satisfies homogeneous Neumann boundary conditions.

For our fixed point iteration, we linearize $f$ about the stable fixed points  $u^{\ast} = \pm 1$, which
satisfy $f'(u^{\ast}) = 0$.  For example, the 
second order scheme from \eqref{eqn:2nd_scheme} becomes
\begin{equation} \label{eqn:AC_2ndorder_iteration}
\left(1 +\Delta t \right)u^{n+1,k+1} =  e^{\epsilon^2 \Delta t \nabla^2} \left( u^n + \frac{\Delta t}{2} f^n \right)+ \frac{\Delta t}{2} \left(f^{n+1,k} +2u^{n+1,k} \right),
\end{equation}
where $n$ indicates the time step as before, and now $k$ is the iteration index. By lagging the nonlinear term $f^{n+1,k}$, the fixed point update is made explicit. Likewise, the
third order scheme from \eqref{eqn:3rd_scheme} becomes

\begin{align} \label{eqn:AC_3rdorder_iteration}
    \left( 1+\frac{2\Delta t}{3} \right)u^{n+1,k+1} 
=  e^{\epsilon^2 \Delta t \nabla^2} \left[u^n +  \frac{2\Delta t}{3} f^n +
\frac{\Delta t}{6} \left( -\epsilon^2 \Delta t \nabla^2 f^n +\Delta t f_t^n \right) \right] +\frac{\Delta t}{3} \left(f^{n+1,k} +2u^{n+1,k}\right).
\end{align}
Here, $e^{\epsilon^2 \Delta t \nabla^2}$ is again understood by replacing it with a resolvent expansion, which is a truncated series of successive convolution operators.

\subsection{Allen-Cahn: One-dimensional test}
\label{subsec:1dallen}

We demonstrate the accuracy of our proposed schemes by simulating a well-known traveling wave solution \cite{chen1998applications,
lee2014semi},

\begin{equation} \label{AC_initial}
u_{AC}(x,t) = \frac{1}{2} \left( 1 - \tanh\left(\frac{x-T_s-st}{2\sqrt{2} \epsilon} \right) \right), \qquad x \in \Omega = [0,4], \quad 0 \leq t \leq T.
\end{equation}
Here, $s = \frac{3\epsilon}{\sqrt{2}} = 0.09 $ is the speed of the traveling
wave, and we choose $\epsilon = 0.03\sqrt{2}$. 
Additionally, we choose the delay time $T_s := 1.5 - sT$, so that the exact solution satisfies 
$u_{AC}(1.5,T) = 0.5$.


\begin{figure}[h]
\centering
\includegraphics[width = 0.3\textwidth]{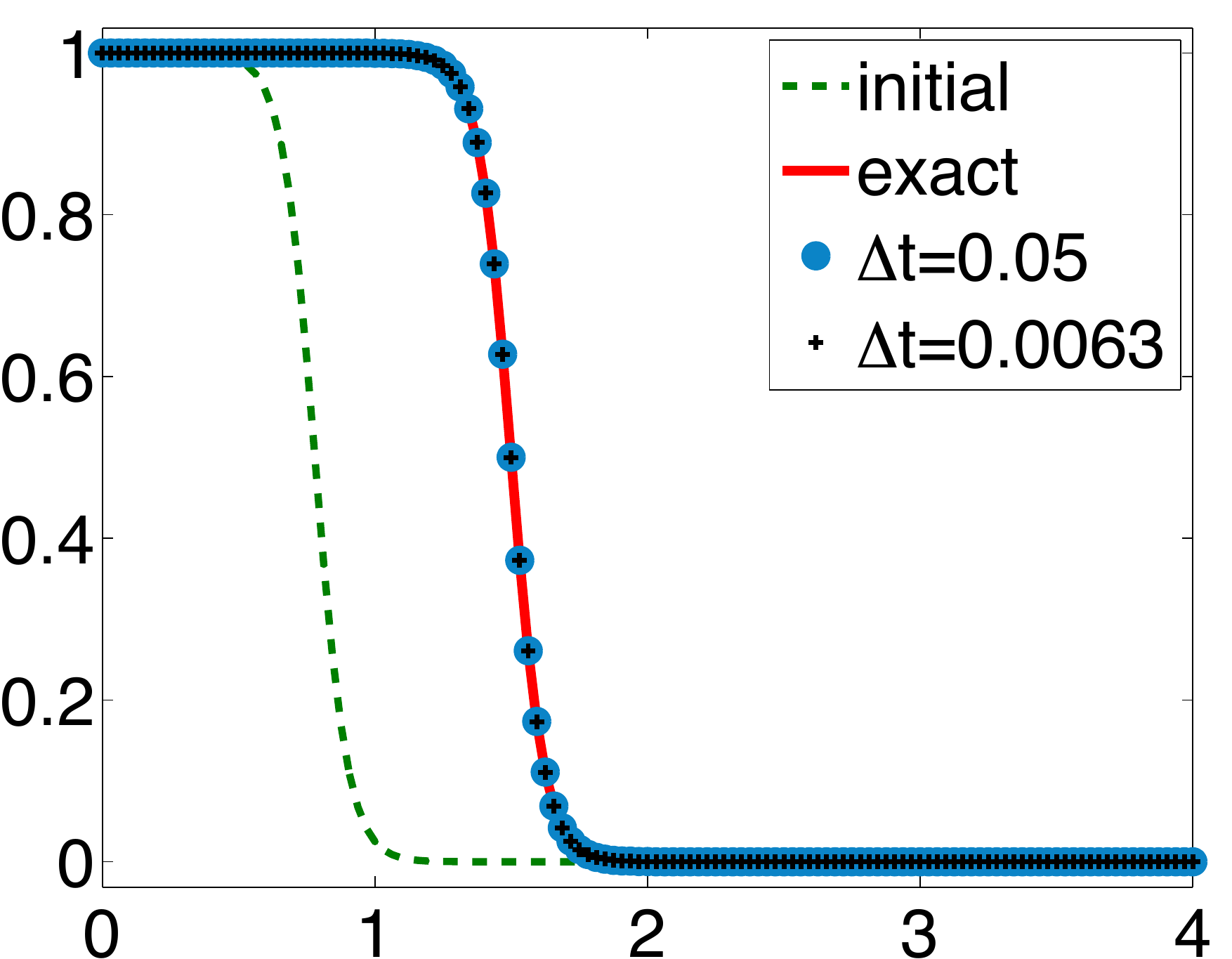}
\caption{Traveling wave solutions $u(x,T)$ at $T=8$ using 
\eqref{eqn:AC_2ndorder_iteration} with two different time step sizes, compared
with the exact profile in \eqref{AC_initial}.}
\label{fig:AC_1d}
\end{figure}

Results for a final time of 
$T=8$ are shown in Figure
\ref{fig:AC_1d}, with two different time steps. The solutions agree well with the exact solution.
\begin{table}[htbp]
\begin{center}
\caption{Refinement study for the 1D Allen-Cahn (AC) equation with an exact traveling wave solution \ref{subsec:1dallen}.}
\label{tab:refinement_AC}
\begin{centering}
\begin{tabular}{|c||c|c||c|c||c|c|}
\hline 
& \multicolumn{2}{c||}{$P=1$}	& \multicolumn{2}{c||}{$P=2$ }				& \multicolumn{2}{c|}{$P=3$}\\ \hline
$\Delta t$		& 	$L^{\infty}$ error	& order & $L^{\infty}$ error		& order 	& $L^{\infty}$ error		& order \\ \hline
$0.025$		& $\num{0.00028216}$	& $-$		& $\num{0.000013895}$	& $-$	 & $\num{0.0000026060}$	& $-$	 \\ \hline
$0.0125$		& $\num{0.00014419}$	& $0.9686$	& $\num{0.0000036115}$	& $1.9439$ & $\num{0.00000039417}$	& $2.7249$	\\ \hline
$0.0063$		& $\num{0.000072874}$	& $0.9845$	& $\num{0.00000092164}$	& $1.9703$ & $\num{0.000000055010}$	& $2.8411$	\\ \hline
$0.0031$	& $\num{0.000036632}$	& $0.9923$	& $\num{0.00000023294}$	& $1.9842$ & $\num{0.0000000073122}$	& $2.9113$	\\ \hline
$0.0016$	& $\num{0.000018365}$	& $0.9961$	& $\num{0.000000058695}$	& $1.9886$ & $\num{0.00000000095714}$	& $2.9335$	\\ \hline
\end{tabular}
\end{centering}
\end{center}
\end{table}
In Table \ref{tab:refinement_AC}, we present the $L^{\infty}$-error in the numerical solution at a final time $T=1$, using the
exact solution $u_{AC}(x,T)$ \eqref{AC_initial}. We observe first order accuracy from the Backward Euler method, and the expected orders of accuracy from the second \eqref{eqn:AC_2ndorder_iteration} and third \eqref{eqn:AC_3rdorder_iteration} order schemes. To ensure that the temporal error is dominant, we have used the fourth order accurate scheme (eq. \eqref{eqn:JL} with $M=4$), with $\Delta x = 2^{-9}$ to perform spatial integration in the successive convolutions.

In principle, we can achieve higher orders accuracy in space and time. The latter would require using higher order
Hermite-Birkhoff interpolation to discretize the reaction term in \eqref{eqn:first}. 

\subsection{Allen-Cahn: Two-dimensional test}
\label{subsec:2dallen}

We next solve the Allen-Cahn equation in two spatial dimensions. A standard benchmark
problem involves the motion of a circular interface 
\cite{chen1998applications, lee2014semi, shen2010numerical}, to which an
exact solution is known in the limiting case $\epsilon \to 0$. The radially symmetric initial conditions are 
defined by
\begin{equation} \label{AC2d_initial}
u(x,y,0) = \tanh \left(
    \frac{0.25 - \sqrt{(x-0.5)^2+(y-0.5)^2}}{\sqrt{2} \epsilon} 
\right),
\end{equation}
which has an interfacial circle ($u(x,y,0)=0$) centered at $(0.5,0.5)$, with a radius of $R_0 = 0.25$. This interfacial circle is unstable, and will shrink over time, as determined by the mean curvature \cite{Allen1979}
\begin{equation}\label{eqn:meancurvature}
V = \frac{dR}{dt} = -\frac{1}{R}.
\end{equation}
Here $V$ is the velocity of the moving interface, and $R$ is the radius of the interfacial circle at time $t$ (i.e., it is the radius of the curve defined by $u(x,y,t)=0$). By integrating \eqref{eqn:meancurvature} with respect to time, we solve for the radius as a function of time
\begin{equation}\label{eqn:velocity}
R(t) = \sqrt{R_0^2 - 2\epsilon^2 t}.
\end{equation}
The location where $\epsilon$ is placed in equation \eqref{eqn:AC} differs from other references
\cite{chen1998applications, lee2014semi, shen2010numerical}.  Therefore, we point out that our time scales have been appropriately rescaled for comparison.

The moving interface problem was simulated using $\epsilon = 0.05$, $\Delta t = \frac{6.4 \times 10^{-4}}{\epsilon^2} =
0.0256$, and $\Delta x = \Delta y = 2^{-8} \approx 0.0039$, which are based on the parameters used in \cite{lee2014semi}. The numerical solution is displayed in Figure \ref{fig:AC_2d}, and we observe that the interfacial circle shrinks, as is expected.

\begin{figure}[h]
\centering
    \subfloat[$u(x,y,0)$]{\includegraphics[width = 0.30\textwidth]{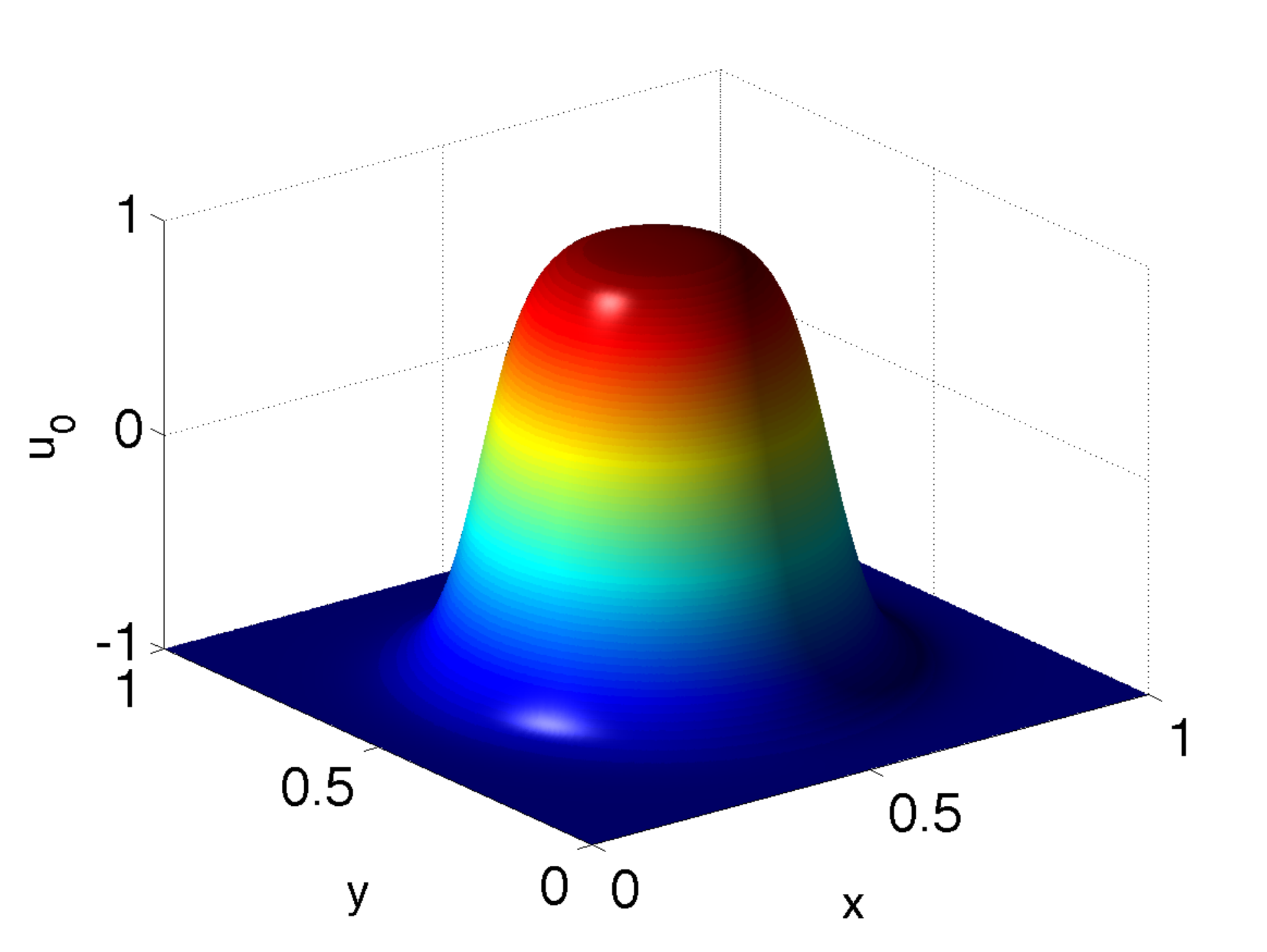}}
    \subfloat[$u(x,y,\frac{T}{2})$]{\includegraphics[width = 0.30\textwidth]{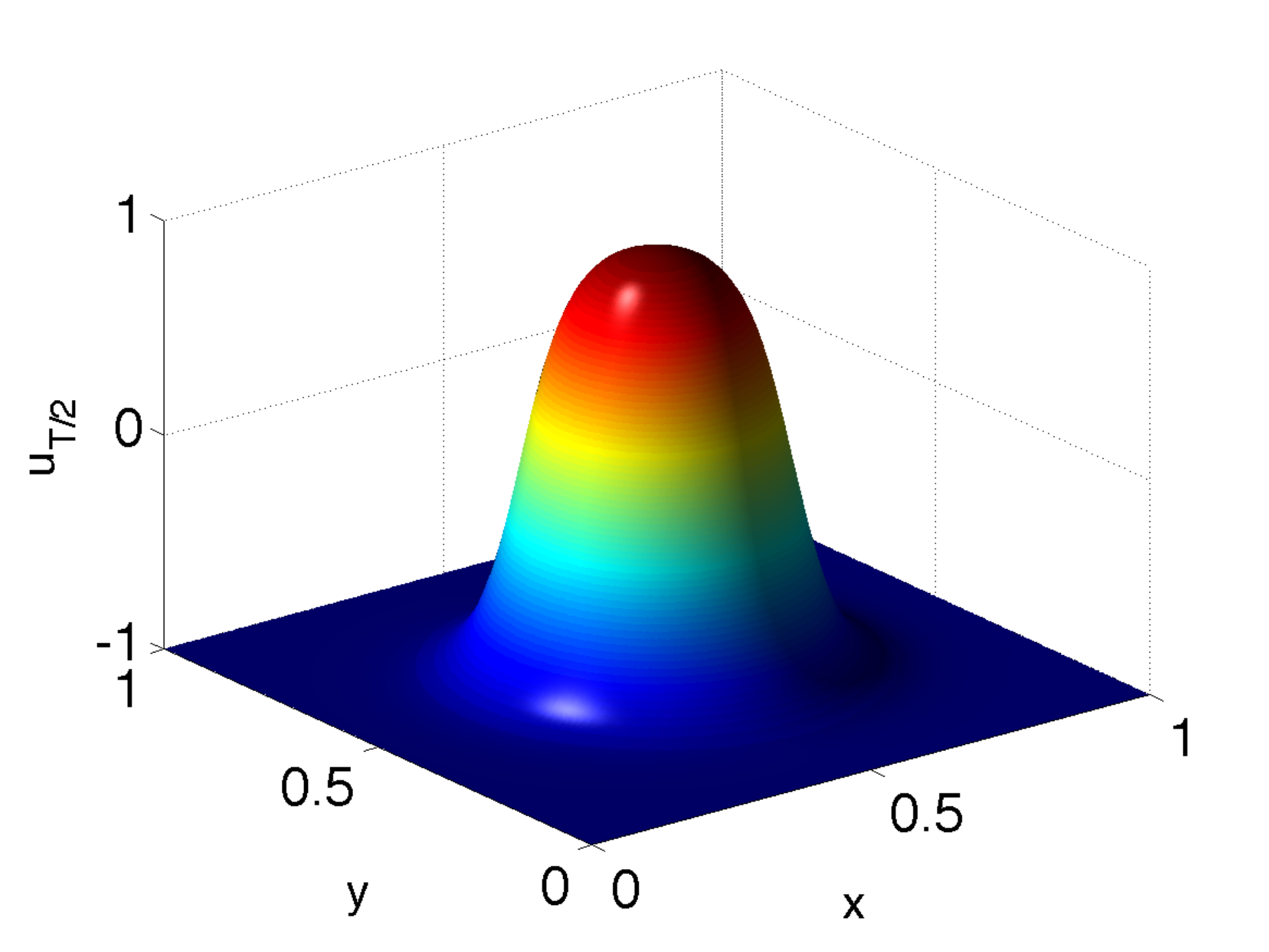}}
    \subfloat[$u(x,y,T)$]{\includegraphics[width = 0.30\textwidth]{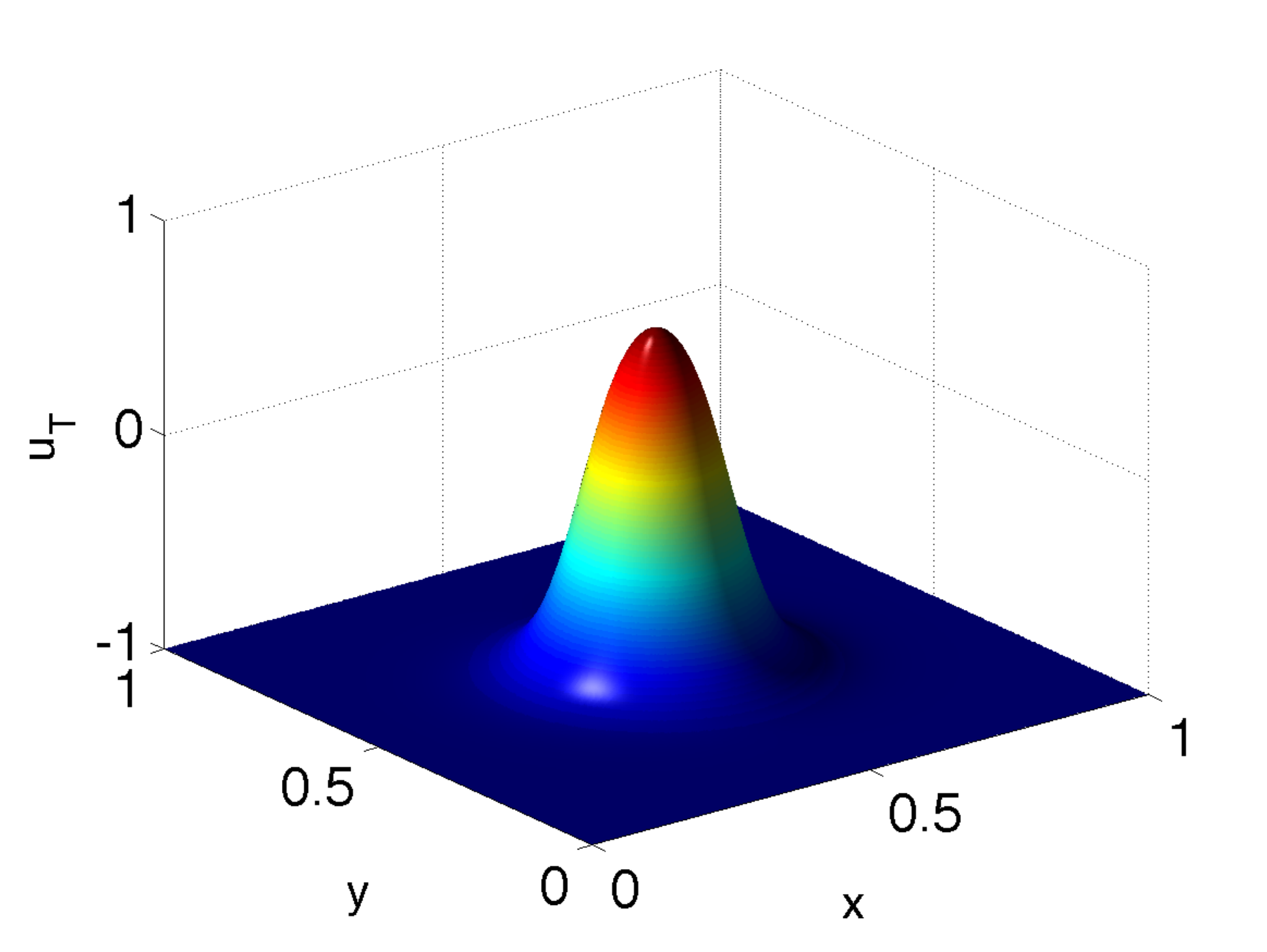}}
    \caption{Time evolution of the initial condition \eqref{AC2d_initial} of the Allen-Cahn equation, up to time $T=\frac{0.0256}{\epsilon^2} = 10.24$.}
    \label{fig:AC_2d}
\end{figure}

In Figure \ref{fig:AC_2d2}, we plot compare the evolution of the radii obtained by our second order scheme with the exact radius \eqref{eqn:velocity}, for two different values of the diffusion parameter $\epsilon$. The radius is measured by taking a slice of the solution along $y=0$, and then solving for the spatial point where $u=0$ using linear interpolation between the two closest
points that satisfy $u(x,0,t) < 0$ and $u(x,0,t) > 0$. Refinement is performed with a fixed spatial mesh $\Delta x = \Delta
y = 2^{-8}$, and time steps of $\Delta t = 0.2560, 0.1280$, and $0.0640$. Because the radius is derived as an exact solution in the limit (i.e., $\epsilon \to 0$) \cite{Allen1979}, we observe that the smaller value of $\epsilon$ is indeed more accurate.

\begin{figure}[h]
\centering
    \subfloat[$\epsilon=0.05$]{\includegraphics[width = 0.46\textwidth]{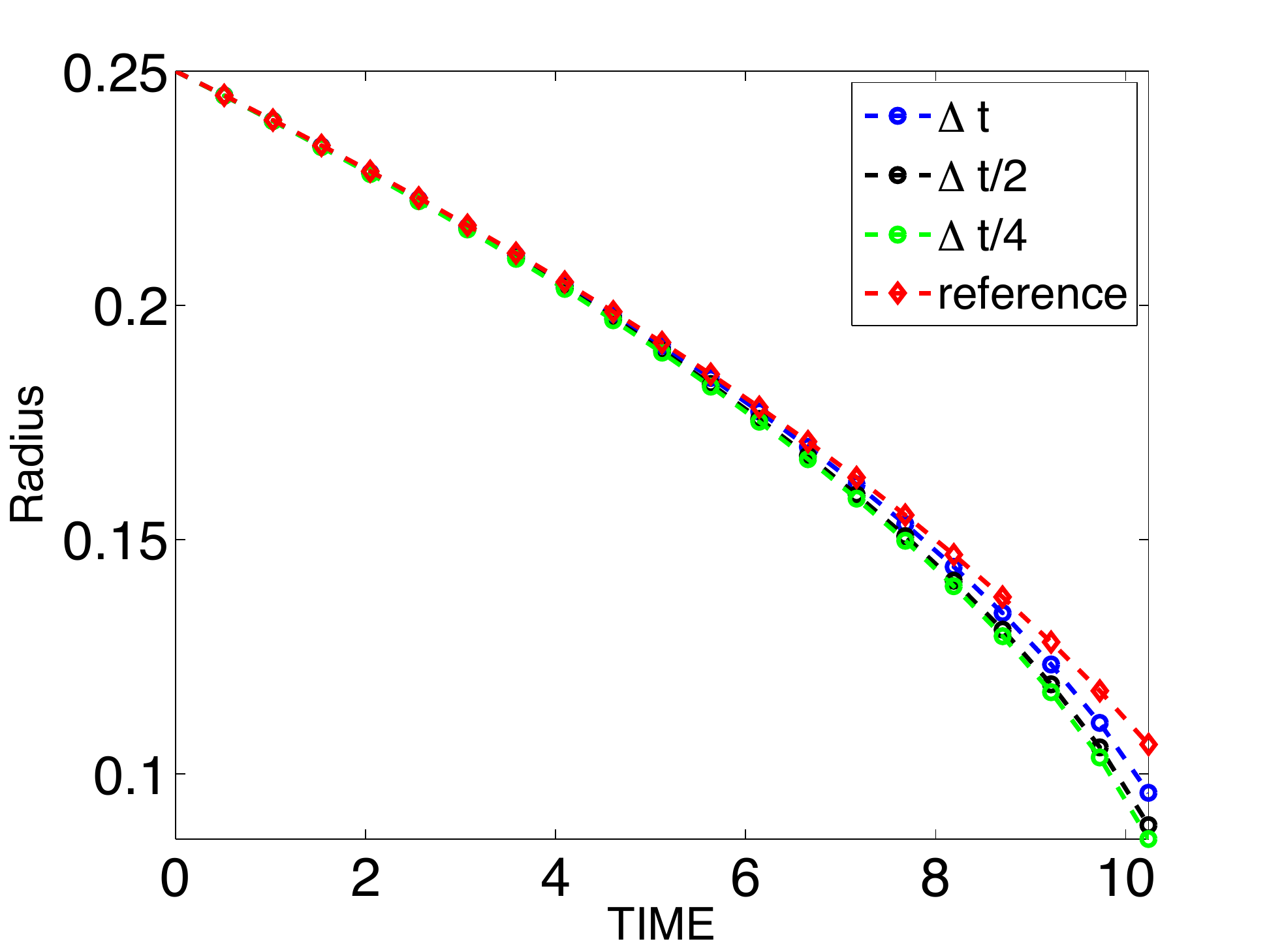}}
    \subfloat[$\epsilon=0.01$]{\includegraphics[width = 0.46\textwidth]{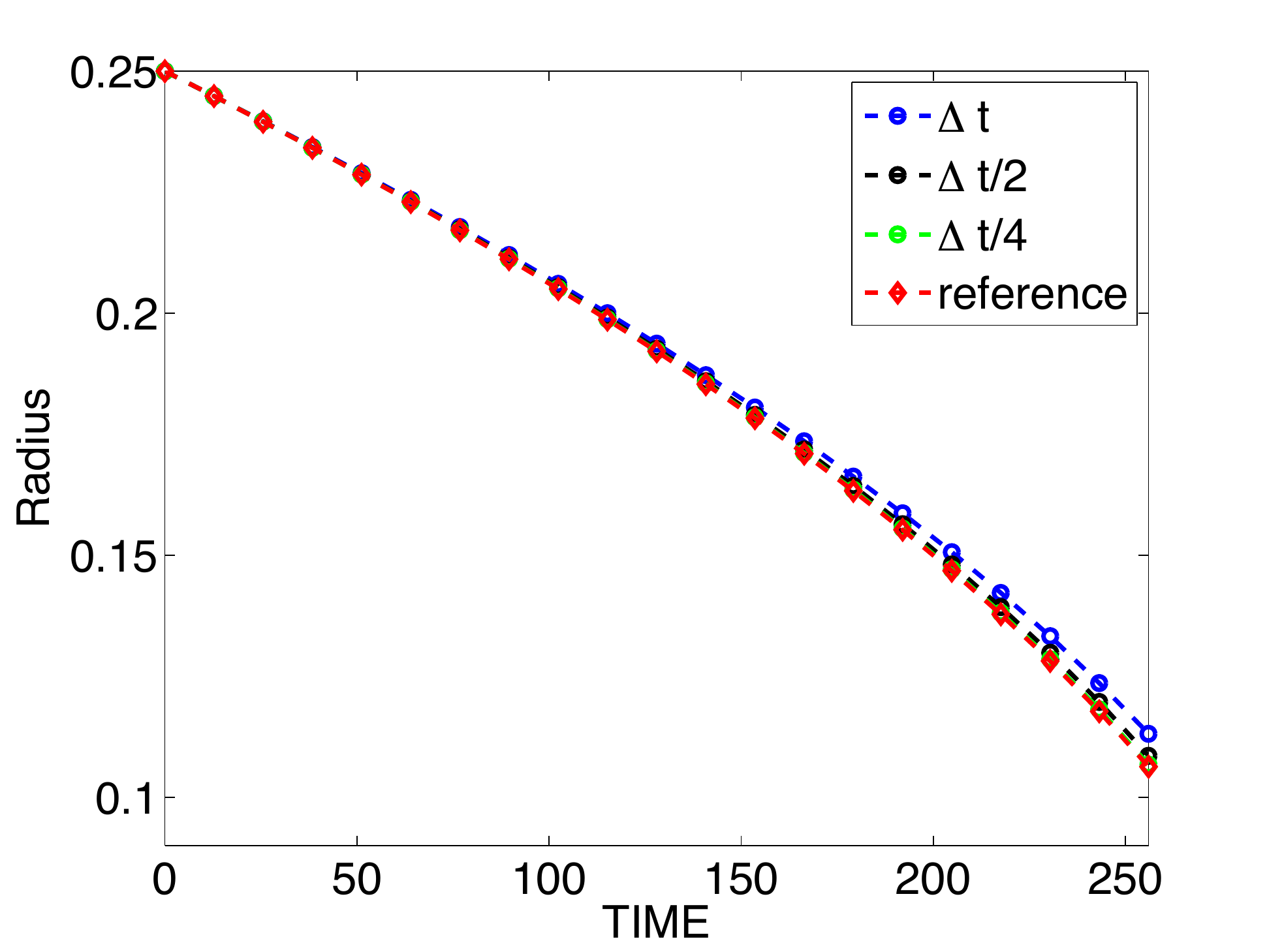}}
    \caption{Radii of the interfacial circle as a function of time ($ 0 \leq t \leq \frac{0.0256}{\epsilon^2}$) compared with  the reference radius R (red line) in \eqref{eqn:velocity}.}
    \label{fig:AC_2d2}
\end{figure}

We next perform a refinement study for the Allen-Cahn equations, but this time in two spatial dimensions. To do so, we must incorporate the multivariate successive convolution algorithms in \eqref{eqn:EPL_2} and \eqref{eqn:Lap_D} into the second 
\eqref{eqn:AC_2ndorder_iteration} and third
\eqref{eqn:AC_3rdorder_iteration} order schemes. Given that we do not have an exact solution, we compute
successive errors in an $L^{\infty}$-norm.  That is, we compute
$||u_{\Delta t} - u_{\frac{\Delta t}{2}} ||_{\infty}$ for each time step
$\dt$. Results are as expected, and are presented in Table \ref{tab:refinement_AC2D}. The parameters used are
$\epsilon = 0.05$, $\Delta x = \Delta y =2^{-9}$, and the final 
computation time is $T=0.5$. Again, the quadrature method is fourth order accurate in space, so that the dominant source of error is temporal.

\begin{table}[htbp]
\begin{center}
\caption{Refinement study for the Allen-Cahn equataion in 2D, with homogeneous Neumann boundary conditions.}
\label{tab:refinement_AC2D}
\begin{centering}
\begin{tabular}{|c||c|c||c|c||c|c|}
\hline 
& \multicolumn{2}{c||}{$P=1$}	& \multicolumn{2}{c||}{$P=2$ }				& \multicolumn{2}{c|}{$P=3$}\\ \hline
$\Delta t$		& 	$L^{\infty}$ error	& order & $L^{\infty}$ error		& order 	& $L^{\infty}$ error		& order \\ \hline
$0.0063$		& $\num{0.00065941}$	& $-$	& $\num{0.00011740}$	& $-$ & $\num{0.0000051744}$	& $-$	\\ \hline
$0.0031$		& $\num{0.00033065}$	& $0.9959$	& $\num{0.000032637}$	& $1.8468$ & $\num{0.00000078351}$	& $2.7234$	\\ \hline
$0.0016$	& $\num{0.00016563}$	& $0.9973$	& $\num{0.0000086726}$	& $1.9120$ & $\num{0.00000010811}$	& $2.8574$	\\ \hline
$0.0008$	& $\num{0.000082894}$	& $0.9987$	& $\num{0.0000022389}$	& $1.9537$ & $\num{0.000000012961}$	& $3.0602$	\\ \hline
\end{tabular}
\end{centering}
\end{center}
\end{table}

\subsection{Two-dimensional test:  The Fitzhugh-Nagumo system}
\label{subsec:fitzhugh-nagumo}

Finally, we solve a well known reaction diffusion system that arises in the modeling of neurons, the Fitzhugh-Nagumo (FHN) model \cite{fife1979mathematical,keener1998mathematical}.  
The FHN system consists of an activator $u$ and an inhibitor $v$, which are coupled via nonlinear reaction diffusion equations

\begin{equation} 
\label{eqn:FHN}
\begin{aligned} 
u_t &= D_{u} \nabla^2 u + \frac{1}{\delta} h(u,v), \\
v_t &= D_{v} \nabla^2 v + g(u,v),
\end{aligned}
\end{equation} 
where $D_u$, $D_v$ are the diffusion coefficients for $u$ and $v$,
respectively, and $0 < \delta \ll 1$ is a real parameter. We use the classical
cubic FHN local dynamics \cite{keener1998mathematical}, that are defined 
as

\begin{equation}
\begin{aligned} 
\label{eqn:dynamics}
h(u,v) &= Cu(1-u)(u-a)-v, \\
g(u,v) &= u - dv,
\end{aligned}
\end{equation}
where $C, a$ and $d$ are dimensionless parameters. The parameters we use are
the same as in \cite{christlieb2015high,olmos2009pseudospectral}: $D_u = 1$,
$D_v =0$, $a=0.1$, $C=1$, $d=0.5$, and $\delta = 0.005$. The diffusion
coefficient for the inhibitor is $D_v=0$, identical to the work found in
\cite{olmos2009pseudospectral,krinsky1998models, starobin1997common,
xie2001coexistence}.



The second order scheme from \eqref{eqn:AC_2ndorder_iteration} is applied to each variable $u$ and $v$ separately.  This defines the numerical scheme as
\begin{equation} 
\begin{aligned} 
u^{n+1} &= e^{\Delta t \nabla^2} \left( u^n + \dfrac{\Delta t}{2\delta} h^n \right) + \dfrac{\Delta t}{2\delta} h^{n+1}, \\
v^{n+1} &= \left( v^n + \frac{\Delta t}{2} g^n \right) + \dfrac{\Delta t}{2} g^{n+1},
\label{eqn:FHN_2ndorder} 
\end{aligned}
\end{equation} 
where $h^n = h(u^n,v^n)$ and $g^n = g(u^n,v^n)$. We again use a stabilized fixed point iteration to address the nonlinear reaction terms. Because $(u^{\ast},v^{\ast}) = (0,0)$ is the only stable excitable fixed point of equation \eqref{eqn:FHN} simply construct the Jacobian of ${\bf{F}} = (h,g)$ of \eqref{eqn:dynamics} about this point:

\begin{equation} \label{eqn:Jacobian}
 \mathcal{J}_{\bf{F}} (u^{\ast},v^{\ast}) \cdot \begin{bmatrix}
       u -u^{\ast}           \\[0.3em]
       v -v^{\ast} 
     \end{bmatrix} \equiv \dfrac{\partial (h,g)}{\partial (u,v)}|_{(0,0)} \cdot \begin{bmatrix}
       u          \\[0.3em]
       v 
     \end{bmatrix}  = \begin{bmatrix}
       -C a& -1            \\[0.3em]
       1 & -d  
     \end{bmatrix} \begin{bmatrix}
       u            \\[0.3em]
       v  
     \end{bmatrix} =  \begin{bmatrix}
       -Cau -v            \\[0.3em]
       u - dv  
     \end{bmatrix}.
\end{equation}
The resulting second order scheme is

\begin{equation} \label{eqn:FHN_2ndorder_iteration}
\begin{bmatrix}
       1+\dfrac{Ca \Delta t}{2 \delta}& \dfrac{\Delta t}{2 \delta}            \\[0.5em]
       -\dfrac{\Delta t}{2} & 1+\dfrac{d \Delta t}{2}  
     \end{bmatrix} \begin{bmatrix}
       u^{n+1,k+1}            \\[0.3em]
       v^{n+1,k+1}  
     \end{bmatrix} = \begin{bmatrix}
       e^{\Delta t \nabla^2} \left( u^n + \dfrac{\Delta t}{2\delta} h^n \right)  \\[0.5em]
       v^n + \dfrac{\Delta t}{2} g^n
     \end{bmatrix} + \dfrac{\Delta t}{2} \begin{bmatrix}
       \frac{1}{\delta} \left( h^{n+1,k} +Ca u^{n+1,k} + v^{n+1,k} \right)         \\[0.5em]
       g^{n+1,k} - u^{n+1,k} + d v^{n+1,k}
     \end{bmatrix},
\end{equation}
where $k$ is the iteration number, and $n$ is the time level.

In Figure \ref{fig:FHN_2d}, we present the numerical evolution of the activator $u$ over the domain $\Omega=[-20,20] \times [-20,20]$,
with periodic boundary conditions. We observe similar spiral waves
that form in other recent work from the literature
\cite{christlieb2015high}.

\begin{figure}[h]
\begin{center}
\centering
    \subfloat[$T=1$]{\includegraphics[width = 0.30\textwidth]{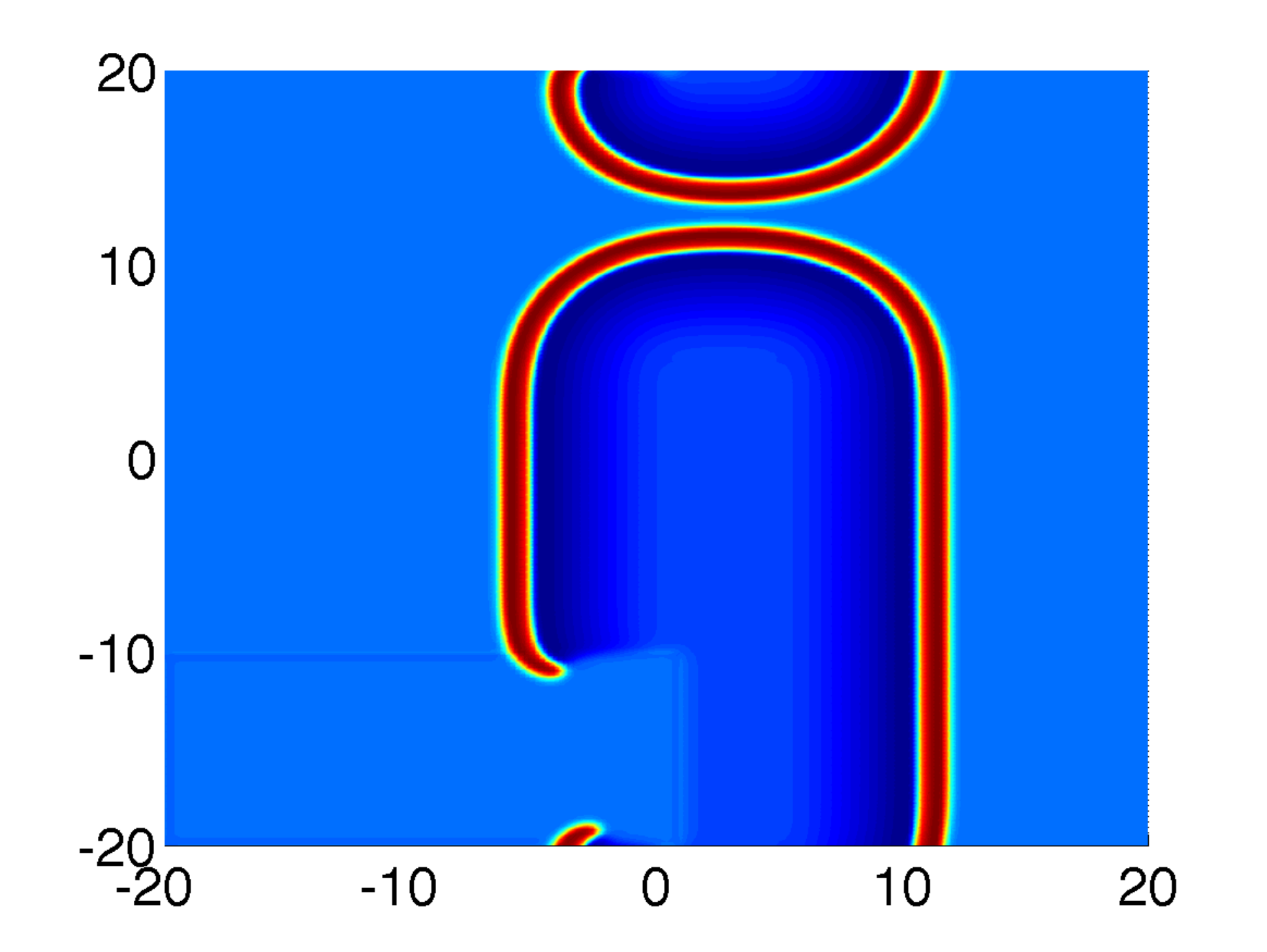}}
    \subfloat[$T=2$]{\includegraphics[width = 0.30\textwidth]{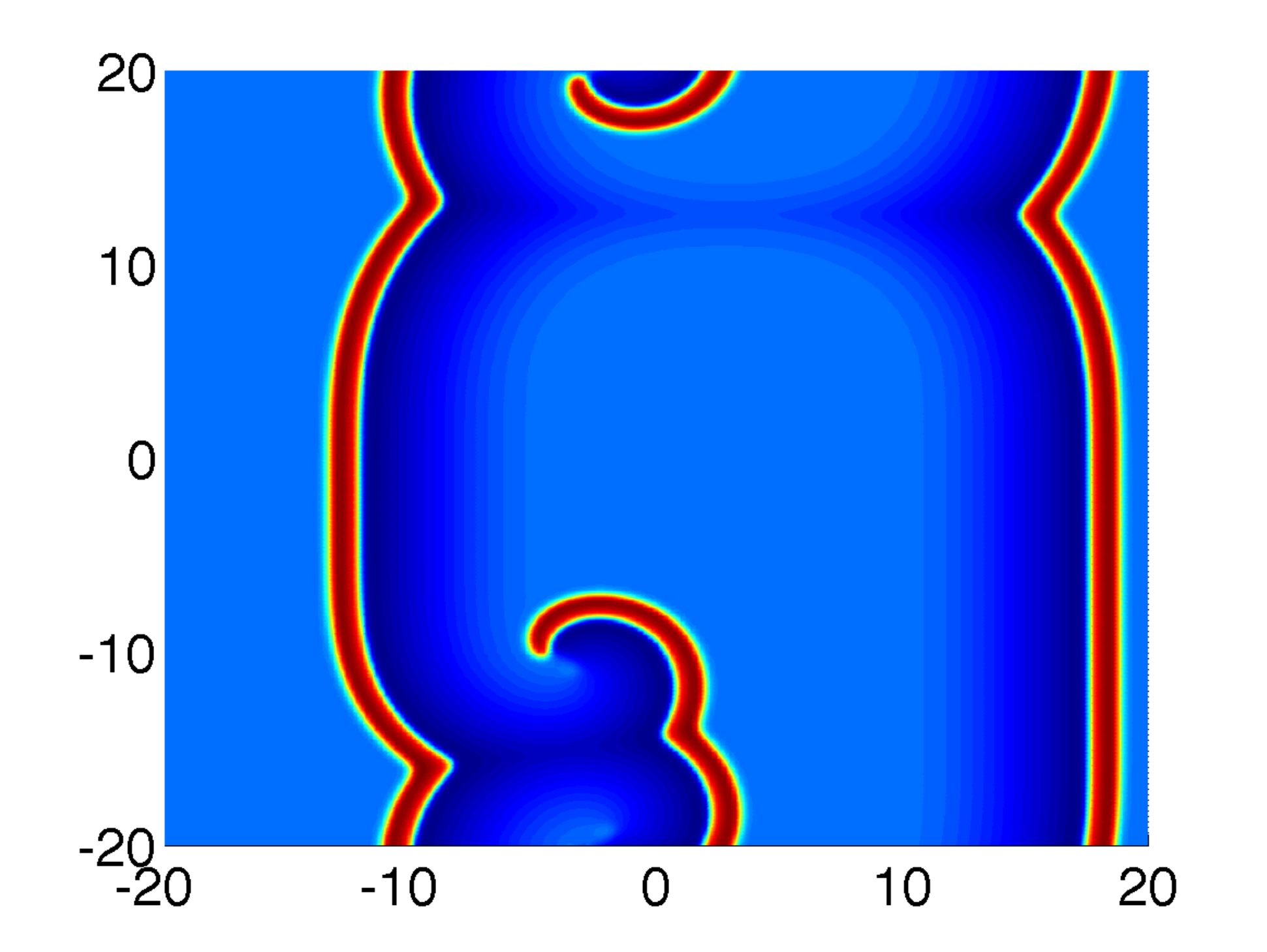}}
    \subfloat[$T=4$]{\includegraphics[width = 0.30\textwidth]{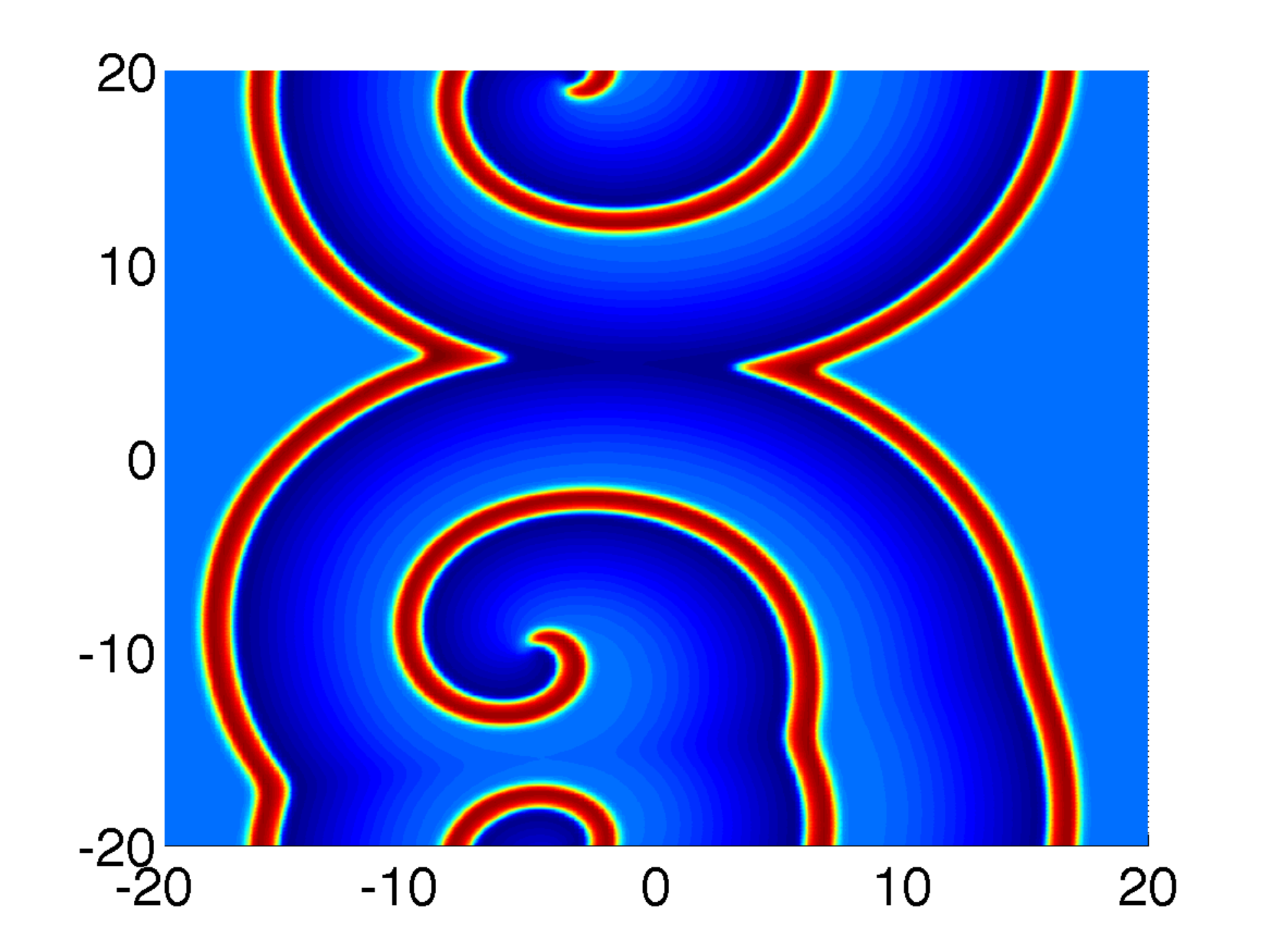}}
    \caption{Temporal evolution of the concentration of activator $u$ using \eqref{eqn:FHN_2ndorder_iteration}.}
    \label{fig:FHN_2d}
\end{center}
\end{figure}

\section{Conclusions}
\label{sec:conclusions}

In this work, we have introduced a numerical scheme for parabolic problems, 
which achieves high order in space and time for the linear heat equation and 
some sample nonlinear reaction diffusion equations. The scheme is $\BigOh(N)$ 
for $N$ spatial points, and exhibits stiff decay for any order of accuracy. 
In the future, we intend to explore parallelization of the algorithm using domain 
decomposition, from which we expect strong performance due to the decoupled spatial 
factorization that we employ.

{\bf Acknowledgements. }  We would like to thank the anonymous reviewer and editor for the encouragement to improve the manuscript.
This work was supported in part by 
AFOSR grants FA9550-12-1-0343, FA9550-12-1-0455, FA9550-15-1-0282, NSF grant
DMS-1418804, New Mexico Consortium grant NMC0155-01, and NASA grant NMX15AP39G.



\bibliographystyle{abbrv}
\bibliography{MOLT_Parabolic}

\begin{thebibliography}{10}

\bibitem{abadias2013c_0}
L.~Abadias and P.~J. Miana.
\newblock {$C_0$}-semigroups and resolvent operators approximated by {L}aguerre
  expansions.
\newblock {\em arXiv preprint arXiv:1311.7542}, 2013.

\bibitem{Allen1979}
S.~Allen and J.~Cahn.
\newblock {A microscopic theory for antiphase boundary motion and its
  application to antiphase domain coarsening}.
\newblock {\em Acta Metallurgica}, 1979.

\bibitem{Cahn1971}
J.~Cahn and J.~Hilliard.
\newblock Spinodal decomposition: A reprise.
\newblock {\em Acta Metallurgica}, 19(2):151 -- 161, 1971.

\bibitem{Cahn1961}
J.~W. Cahn.
\newblock On spinodal decomposition.
\newblock {\em Acta Metallurgica}, 9(9):795 -- 801, 1961.

\bibitem{Causley_2013c}
M.~Causley, A.~Christlieb, B.~Ong, and L.~Van~Groningen.
\newblock Method of lines transpose: an implicit solution to the wave equation.
\newblock {\em Math. Comp.}, 83(290):2763--2786, 2014.

\bibitem{Causley_2013}
M.~F. Causley and A.~J. Christlieb.
\newblock Higher order {A}-stable schemes for the wave equation using a
  successive convolution approach.
\newblock {\em SIAM J. Numer. Anal.}, 52(1):220--235, 2014.

\bibitem{Causley_2013b}
M.~F. Causley, A.~J. Christlieb, Y.~Guclu, and E.~Wolf.
\newblock Method of lines transpose: {A} fast implicit wave propagator.
\newblock {\em arXiv preprint arXiv:1306.6902}, 2013.

\bibitem{chen1998applications}
L.~Chen and J.~Shen.
\newblock Applications of semi-implicit {F}ourier-spectral method to phase
  field equations.
\newblock {\em Computer Physics Communications}, 108(2):147--158, 1998.

\bibitem{seal2014picard}
A.~J. Christlieb, Y.~G{\"u}{\c{c}}l{\"u}, and D.~C. Seal.
\newblock The {P}icard integral formulation of weighted essentially
  non-oscillatory schemes.
\newblock {\em SIAM J. Numer. Anal.}, 2015.
\newblock (accepted).

\bibitem{christlieb2015high}
A.~J. Christlieb, Y.~Liu, and Z.~Xu.
\newblock High order operator splitting methods based on an integral deferred
  correction framework.
\newblock {\em Journal of Computational Physics}, 294:224--242, 2015.

\bibitem{Crank1996}
J.~Crank and P.~Nicolson.
\newblock A practical method for numerical evaluation of solutions of partial
  differential equations of the heat-conduction type.
\newblock {\em Advances in Computational Mathematics}, 6(3):207--226, 1996.

\bibitem{Dai2013}
S.~Dai and K.~Promislow.
\newblock {Geometric evolution of bilayers under the functionalized
  Cahn-Hilliard equation}.
\newblock {\em Proceedings of the Royal Society A: Mathematical, Physical and
  Engineering Science}, 469(2153), 2013.

\bibitem{Douglas1955}
J.~Douglas, Jr.
\newblock On the numerical integration of $\frac{\partial ^2 u}{\partial x^2 }
  + \frac{\partial ^2 u}{\partial y^2 } = \frac{\partial u}{\partial t}$ by
  implicit methods.
\newblock {\em Journal of the Society for Industrial and Applied Mathematics},
  3(1):42--65, 1955.

\bibitem{Douglas1962}
J.~Douglas, Jr.
\newblock Alternating direction methods for three space variables.
\newblock {\em Numerische Mathematik}, 4(1):41--63, 1962.

\bibitem{Douglas1956}
J.~Douglas, Jr. and H.~Rachford.
\newblock On the numerical solution of heat conduction problems in two and
  three space variables.
\newblock {\em Transactions of the American mathematical Society},
  82(2):421--439, 1956.

\bibitem{Fairweather1967}
G.~Fairweather and A.~Mitchell.
\newblock A new computational procedure for {ADI} methods.
\newblock {\em SIAM Journal on Numerical Analysis}, 4(2), 1967.

\bibitem{fife1979mathematical}
P.~C. Fife et~al.
\newblock {\em Mathematical aspects of reacting and diffusing systems.}
\newblock Springer Verlag., 1979.

\bibitem{Fisher1999}
R.~A. Fisher.
\newblock {\em The genetical theory of natural selection: a complete variorum
  edition}.
\newblock Oxford University Press, 1999.

\bibitem{Fitzhugh1961}
R.~FitzHugh.
\newblock Impulses and physiological states in theoretical models of nerve
  membrane.
\newblock {\em Biophysical journal}, 1(6):445--466, 1961.

\bibitem{Greengard1991}
L.~Greengard and J.~Strain.
\newblock The fast {G}auss transform.
\newblock {\em SIAM J. Sci. Statist. Comput.}, 12(1):79--94, 1991.

\bibitem{grimm2010approximation}
V.~Grimm and M.~Gugat.
\newblock Approximation of semigroups and related operator functions by
  resolvent series.
\newblock {\em SIAM Journal on Numerical Analysis}, 48(5):1826--1845, 2010.

\bibitem{Jia2008}
J.~Jia and J.~Huang.
\newblock {Krylov deferred correction accelerated method of lines transpose for
  parabolic problems}.
\newblock {\em Journal of Computational Physics}, 227(3):1739--1753, Jan. 2008.

\bibitem{Jiang2013}
S.~Jiang, L.~Greengard, and S.~Wang.
\newblock {Efficient sum-of-exponentials approximations for the heat kernel and
  their applications}.
\newblock {\em arXiv preprint arXiv:1308.3883}, pages 1--23, 2013.

\bibitem{Kassam2005}
A.~Kassam and L.~Trefethen.
\newblock {Fourth-order time-stepping for stiff PDEs}.
\newblock {\em SIAM Journal on Scientific Computing}, 26(4):1214--1233, 2005.

\bibitem{keener1998mathematical}
J.~P. Keener and J.~Sneyd.
\newblock {\em Mathematical physiology}, volume~1.
\newblock Springer, 1998.

\bibitem{krinsky1998models}
V.~Krinsky and A.~Pumir.
\newblock Models of defibrillation of cardiac tissue.
\newblock {\em Chaos: An Interdisciplinary Journal of Nonlinear Science},
  8(1):188--203, 1998.

\bibitem{Kropinski2011}
M.~C.~A. Kropinski and B.~D. Quaife.
\newblock Fast integral equation methods for the modified {H}elmholtz equation.
\newblock {\em J. Comput. Phys.}, 230(2):425--434, 2011.

\bibitem{Lambers2008}
J.~V. Lambers.
\newblock Implicitly defined high-order operator splittings for parabolic and
  hyperbolic variable-coefficient {PDE} using modified moments.
\newblock {\em Int. J. Pure Appl. Math.}, 50(2):239--244, 2009.

\bibitem{lee2014semi}
H.~G. Lee and J.-Y. Lee.
\newblock A semi-analytical {F}ourier spectral method for the {A}llen-{C}ahn
  equation.
\newblock {\em Comput. Math. Appl.}, 68(3):174--184, 2014.

\bibitem{Li2009}
J.~Li and L.~Greengard.
\newblock {High order accurate methods for the evaluation of layer heat
  potentials}.
\newblock {\em SIAM Journal on Scientific Computing}, 31(5):3847--3860, 2009.

\bibitem{Lubich1992}
C.~Lubich and R.~Schneider.
\newblock {Time discretization of parabolic boundary integral equations}.
\newblock {\em Numerische Mathematik}, 455481, 1992.

\bibitem{Lyon2010}
M.~Lyon and O.~Bruno.
\newblock {High-order unconditionally stable FC-AD solvers for general smooth
  domains II. Elliptic, parabolic and hyperbolic PDEs; theoretical
  considerations}.
\newblock {\em Journal of Computational Physics}, 229(9):3358--3381, 2010.

\bibitem{Bruno2010}
M.~Lyon and O.~P. Bruno.
\newblock High-order unconditionally stable {FC}-{AD} solvers for general
  smooth domains. {II}. {E}lliptic, parabolic and hyperbolic {PDE}s;
  theoretical considerations.
\newblock {\em J. Comput. Phys.}, 229(9):3358--3381, 2010.

\bibitem{Nagumo1962}
J.~Nagumo, S.~Arimoto, and S.~Yoshizawa.
\newblock An active pulse transmission line simulating nerve axon.
\newblock {\em Proceedings of the IRE}, 50(10):2061--2070, 1962.

\bibitem{Norsett1974}
S.~P. N{\o}rsett.
\newblock {One-step methods of Hermite type for numerical integration of stiff
  systems}.
\newblock {\em BIT Numerical Mathematics}, 14, 1974.

\bibitem{olmos2009pseudospectral}
D.~Olmos and B.~D. Shizgal.
\newblock Pseudospectral method of solution of the {F}itzhugh-{N}agumo
  equation.
\newblock {\em Math. Comput. Simulation}, 79(7):2258--2278, 2009.

\bibitem{Peaceman1955}
D.~W. Peaceman and H.~H. Rachford, Jr.
\newblock The numerical solution of parabolic and elliptic differential
  equations.
\newblock {\em J. Soc. Indust. Appl. Math.}, 3:28--41, 1955.

\bibitem{Rothe32}
E.~Rothe.
\newblock Zweidimensionale parabolische {R}andwertaufgaben als {G}renzfall
  eindimensionaler {R}andwertaufgaben.
\newblock {\em Math. Ann.}, 102(1):650--670, 1930.

\bibitem{Salazar2000}
A.~J. Salazar, M.~Raydan, and A.~Campo.
\newblock Theoretical analysis of the exponential transversal method of lines
  for the diffusion equation.
\newblock {\em Numer. Methods Partial Differential Equations}, 16(1):30--41,
  2000.

\bibitem{seal2014high}
D.~C. Seal, Y.~G{\"u}{\c{c}}l{\"u}, and A.~J. Christlieb.
\newblock High-order multiderivative time integrators for hyperbolic
  conservation laws.
\newblock {\em Journal of Scientific Computing}, 60(1):101--140, 2014.

\bibitem{shen2010numerical}
J.~Shen and X.~Yang.
\newblock Numerical approximations of {A}llen-{C}ahn and {C}ahn-{H}illiard
  equations.
\newblock {\em Discrete Contin. Dyn. Syst.}, 28(4):1669--1691, 2010.

\bibitem{Shreve2004}
S.~Shreve.
\newblock {\em Stochastic Calculus for Finance II: Continuous-Time Models}.
\newblock Number v. 11 in Springer Finance. Springer, 2004.

\bibitem{starobin1997common}
J.~Starobin and C.~Starmer.
\newblock Common mechanism links spiral wave meandering and
  wave-front--obstacle separation.
\newblock {\em Physical Review E}, 55(1):1193, 1997.

\bibitem{Tausch2007}
J.~Tausch.
\newblock {A fast method for solving the heat equation by layer potentials}.
\newblock {\em Journal of Computational Physics}, 224(2):956--969, June 2007.

\bibitem{xie2001coexistence}
F.~Xie, Z.~Qu, J.~N. Weiss, and A.~Garfinkel.
\newblock Coexistence of multiple spiral waves with independent frequencies in
  a heterogeneous excitable medium.
\newblock {\em Physical Review E}, 63(3):031905, 2001.

\end{thebibliography}

\end{document}